\documentclass[a4paper]{amsart}

\usepackage[margin=1in]{geometry}

\usepackage[english]{babel}
\usepackage[utf8]{inputenc}
\usepackage{amsmath,amssymb}
\usepackage{amsthm}
\usepackage{amsfonts}
\usepackage{asymptote}
\usepackage{array}
\usepackage{comment}
\usepackage{graphicx}
\usepackage{hyperref}
\usepackage{mathabx,mathtools}
\usepackage{multirow}
\usepackage{textcomp}
\usepackage[colorinlistoftodos]{todonotes}
\usepackage[enableskew,vcentermath]{youngtab}
\usepackage[all]{xy}

\newtheorem{theorem}{Theorem}[section]
\newtheorem{proposition}[theorem]{Proposition}
\newtheorem{lemma}[theorem]{Lemma}
\newtheorem{lemdef}[theorem]{Lemma/Definition}
\newtheorem{corollary}[theorem]{Corollary}
\newtheorem{conjecture}[theorem]{Conjecture}

\theoremstyle{definition}

\newtheorem{definition}[theorem]{Definition}
\newtheorem{remark}[theorem]{Remark}
\newtheorem{example}[theorem]{Example}
\newtheorem{question}[theorem]{Question}

\newcommand{\newword}[1]{\textbf{\emph{#1}}}

\newcommand{\sh}{\mathrm{sh}}
\newcommand{\rsh}{\mathrm{rsh}}
\newcommand{\rectify}{\mathrm{rect}}
\newcommand{\ev}{\mathrm{ev}}
\newcommand{\esh}{\mathrm{esh}}
\newcommand{\lesh}{\mathrm{local\text{-}esh}}
\newcommand{\eset}{\varnothing}
\newcommand{\rect}{{\scalebox{.3}{\yng(3,3)}}}
\newcommand{\x}{\times}
\newcommand{\ybox}{\boxtimes}
\newcommand{\ebox}{\scalebox{.5}{\yng(1)}}

\newcommand{\pieri}{{\bf \mathrm{Pieri}}}
\newcommand{\vertical}{{\bf \mathrm{Vert}}}
\newcommand{\jump}{{\bf \mathrm{Jump}}}

\newcommand{\horiz}{{\bf \mathrm{Horiz}}}
\newcommand{\cpieri}{{\bf \mathrm{CPieri}}}

\newcommand{\LR}{\mathrm{LR}}
\newcommand{\HW}{\mathrm{LR}}
\newcommand{\DE}{\mathrm{DE}}
\newcommand{\DEyb}{\DE(\alpha,{\scalebox{.5}{\yng(1)}},\beta, \gamma)}
\newcommand{\DEby}{\DE(\alpha,\beta,{\scalebox{.5}{\yng(1)}}, \gamma)}
\newcommand{\LRyb}{\LR(\alpha,{\scalebox{.5}{\yng(1)}},\beta, \gamma)}
\newcommand{\LRby}{\LR(\alpha,\beta,{\scalebox{.5}{\yng(1)}}, \gamma)}
\newcommand{\Kabg}{K(\gamma^c/\alpha; \beta)}

\DeclareMathOperator{\id}{id}

\title{Monodromy and K-theory of Schubert curves \\ via generalized jeu de taquin}

\author{Maria Monks Gillespie}
\address[Maria Monks Gillespie]{
Mathematics Department\\
University of California\\
Davis, CA}
\email{mgillespie@math.ucdavis.edu}

\author{Jake Levinson}
\address[Jake Levinson]{
Mathematics Department\\
University of Michigan\\
Ann Arbor, MI}
\email{jakelev@umich.edu}

\thanks{The first author was supported by NSF GRFP and the Hertz Foundation. The second author was supported by FRQNT and by NSF grants DMS-1160720, 1464693 and 1101152. \bigskip \\ The final publication is available at Springer via http://dx.doi.org/10.1007/s10801-016-0705-7}

\date{\today}

\begin{document}

\begin{abstract}
%%%%%%%%%%%%%%%%%%%%%%%%%
%%%%% ABSTRACT      %%%%%
%%%%%%%%%%%%%%%%%%%%%%%%%

We establish a combinatorial connection between the real geometry and the K-theory of complex \textit{Schubert curves} $S(\lambda_\bullet)$, which are one-dimensional Schubert problems defined with respect to flags osculating the rational normal curve.  In \cite{bib:Levinson}, it was shown that the real geometry of these curves is described by the orbits of a map $\omega$ on skew tableaux, defined as the commutator of jeu de taquin rectification and promotion.  In particular, the real locus of the Schubert curve is naturally a covering space of $\mathbb{RP}^1$, with $\omega$ as the monodromy operator.

We provide a fast, local algorithm for computing $\omega$ without rectifying the skew tableau, and show that certain steps in our algorithm are in bijective correspondence with Pechenik and Yong's \emph{genomic tableaux} \cite{bib:Pechenik}, which enumerate the K-theoretic Littlewood-Richardson coefficient associated to the Schubert curve. We then give purely combinatorial proofs of several numerical results involving the K-theory and real geometry of $S(\lambda_\bullet)$.

%Let $T$ be a skew Young tableau containing a single zero, but which is ballot for its other entries. In this paper we study the commutator $\omega$ of rectifying and promoting $T$. In \cite{bib:Levinson}, it was shown that the orbit structure of $\omega$ describes the real geometry of certain one-dimensional Schubert problems defined with respect to flags osculating the rational normal curve at real points. All such Schubert problems are covering spaces of $\mathbb{RP}^1$, with $\omega$ as the monodromy operator.

%Our main result is an algorithm, similar to jeu de taquin, which computes $\omega(T)$ without rectifying the skew shape.
% make this sound better
%We use our simpler description of $\omega$ to study these curves, particularly from the perspective of K-theoretic Schubert calculus. We show that our algorithm generates each corresponding genomic tableau exactly once, and we give purely combinatorial proofs of several numerical facts about $\omega$ exhibited from the associated geometry.
\end{abstract}

\maketitle

\section{Introduction}\label{sec:introduction}
%%%%%%%%%%%%%%%%%%%%%%%%%
%%%%% INTRODUCTION  %%%%%
%%%%%%%%%%%%%%%%%%%%%%%%%
In this paper, we study the real and complex geometry of certain one-dimensional intersections $S$ of Schubert varieties defined with respect to `osculating' flags.  
%These curves are known to have smooth real points, which naturally cover the circle $\mathbb{RP}^1$; we give a new combinatorial rule, in terms of certain Young tableaux, for the monodromy operator on the fibers (a more complicated rule was given in \cite{bib:Levinson}). Our rule is fast and combinatorially `local', making it easier to count $\eta(S)$, the number of components of $S(\mathbb{R})$, which fully characterizes the real topology.  Moreover, our rule computes the class of $S$ in the $K$-theory of the Grassmannian: it explicitly produces Pechenik and Yong's \emph{genomic tableaux} \cite{bib:Pechenik}. Using this connection, we give purely combinatorial proofs of two known geometric relations between $\eta(S)$ and the Euler characteristic $\chi(\mathcal{O}_S)$. We also find new facts about the real and complex geometry of $S$.
To define $S$, recall first that the \emph{rational normal curve} is the image of the Veronese embedding $\mathbb{P}^1 \hookrightarrow \mathbb{P}^{n-1} = \mathbb{P}(\mathbb{C}^n)$, defined by
\[t \mapsto [1 : t : t^2 : \cdots : t^{n-1}].\]
Let $\mathcal{F}_t$ be the \emph{osculating} or \emph{maximally tangent flag} to this curve at $t \in \mathbb{P}^1$,  i.e. the complete flag in $\mathbb{C}^n$ formed by the iterated derivatives of this map. The $i$-th part of the flag is spanned by the top $i$ rows of the matrix
\begin{equation*} \label{eqn:flag-matrix}
\begin{bmatrix}
\big(\frac{\mathrm{d}}{\mathrm{d}t}\big)^{i-1}(t^{j-1})
\end{bmatrix}
=
\begin{bmatrix}
1 & t & t^2 & \cdots & t^{n-1} \\
0 & 1 & 2t & \cdots & (n-1) t^{n-2} \\
0 & 0 & 2 & \cdots & (n-1)(n-2) t^{n-3} \\
\vdots & \vdots & \vdots &\ddots & \vdots \\
0 & 0 & 0 & \cdots & (n-1)!
\end{bmatrix}.
\end{equation*}
Let $G(k,\mathbb{C}^n)$ be the Grassmannian, and $\mathrm{\Omega}(\lambda,\mathcal{F}_t)$ the Schubert variety for the condition $\lambda$ with respect to $\mathcal{F}_t$. The \newword{Schubert curve} is the intersection \[S = S(\lambda^{(1)}, \ldots, \lambda^{(r)}) = \mathrm{\Omega}(\lambda^{(1)}, \mathcal{F}_{t_1}) \cap \cdots \cap \mathrm{\Omega}(\lambda^{(r)}, \mathcal{F}_{t_r}),\]
where the osculation points $t_i$ are real numbers with $0 = t_1 < t_2 < \cdots < t_r = \infty$, and $\lambda^{(1)},\ldots,\lambda^{(r)}$ are partitions for which $\sum |\lambda^{(i)}|=k(n-k)-1$.   For simplicity, we always consider intersections of only three Schubert varieties, though the results of this paper (in particular, Theorems \ref{thm:intro-2}, \ref{thm:MainResult2} and \ref{thm:intro-parity}) extend to the general case without difficulty.  With this in mind, we consider a triple of partitions $\alpha, \beta, \gamma$ with $|\alpha| + |\beta| + |\gamma| = k(n-k) - 1$, and we study the Schubert curve
\[S(\alpha,\beta,\gamma) = \mathrm{\Omega}(\alpha,\mathcal{F}_0) \cap \mathrm{\Omega}(\beta,\mathcal{F}_1) \cap \mathrm{\Omega}(\gamma,\mathcal{F}_\infty).\]

Schubert varieties with respect to osculating flags have been studied extensively in the context of degenerations of curves \cite{bib:Chan} \cite{bib:EH86} \cite{bib:Oss06}, Schubert calculus and the Shapiro-Shapiro Conjecture \cite{bib:MTV09} \cite{bib:Pur13} \cite{bib:Sot10}, and the geometry of the moduli space $\overline{M_{0,r}}(\mathbb{R})$ \cite{bib:Speyer}. They satisfy unusually strong transversality properties, particularly under the hypothesis that the osculation points $t$ are real \cite{bib:EH86} \cite{bib:MTV09}; in particular, $S$ is known to be one-dimensional (if nonempty) and reduced \cite{bib:Levinson}. Moreover, intersections of such Schubert varieties in dimensions zero and one have been found to have remarkable topological descriptions in terms of Young tableau combinatorics. \cite{bib:Chan} \cite{bib:Levinson} \cite{bib:Pur10} \cite{bib:Speyer}

The Schubert curve is no exception: recent work \cite{bib:Levinson} has shown that its \emph{real} connected components can be described by combinatorial operations, related to jeu de taquin and Sch\"{u}tzenberger's promotion and evacuation, on chains of skew Young tableaux. Recall that a skew semistandard Young tableau is \emph{Littlewood-Richardson} if its reading word is \emph{ballot}, meaning that every suffix of the reading word has partition content.
  
 \begin{definition}\label{def:chains}
  We write $\mathrm{LR}(\lambda^{(1)},\ldots,\lambda^{(r)})$ to denote the set of sequences $(T_1, \ldots, T_r)$ of skew Littlewood-Richardson tableaux, filling a $k\times (n-k)$ rectangle, such that the shape of $T_i$ extends that of $T_{i-1}$ and $T_i$ has content $\lambda^{(i)}$ for all $i$. (The tableaux $T_1$ and $T_r$ are uniquely determined and may be omitted.)
\end{definition}

The theorem below describes the topology of $S(\alpha,\beta,\gamma)(\mathbb{R})$ in terms of tableaux:

\begin{theorem}[\cite{bib:Levinson}, Corollary 4.9]\label{thm:intro-2}
There is a map $S \to \mathbb{P}^1$ that makes the real locus $S(\mathbb{R})$ a smooth covering of the circle $\mathbb{RP}^1$ (Figure \ref{fig:covering-space}). The fibers over $0$ and $\infty$ are in canonical bijection with, respectively, $\LRyb$ and $\LRby$. Under this identification, the arcs of $S(\mathbb{R})$ covering $\mathbb{R}_-$ induce the \emph{jeu de taquin bijection}
\[\sh : \LRby \to \LRyb,\]
and the arcs covering $\mathbb{R}_+$ induce a different bijection $\esh$, called \emph{evacuation-shuffling}. The monodromy operator $\omega$ is, therefore, given by $\omega = \sh \circ \esh.$
%\begin{itemize}
%\item 
%\item The arcs of $S(\mathbb{R})$ covering $\mathbb{R}_+$ induce a different bijection, $\esh$, called \emph{evacuation shuffling}.
%\end{itemize}
\end{theorem}

%\vspace{-0.2cm}

\begin{figure}[ht]
\centering
\qquad \includegraphics[scale=0.57]{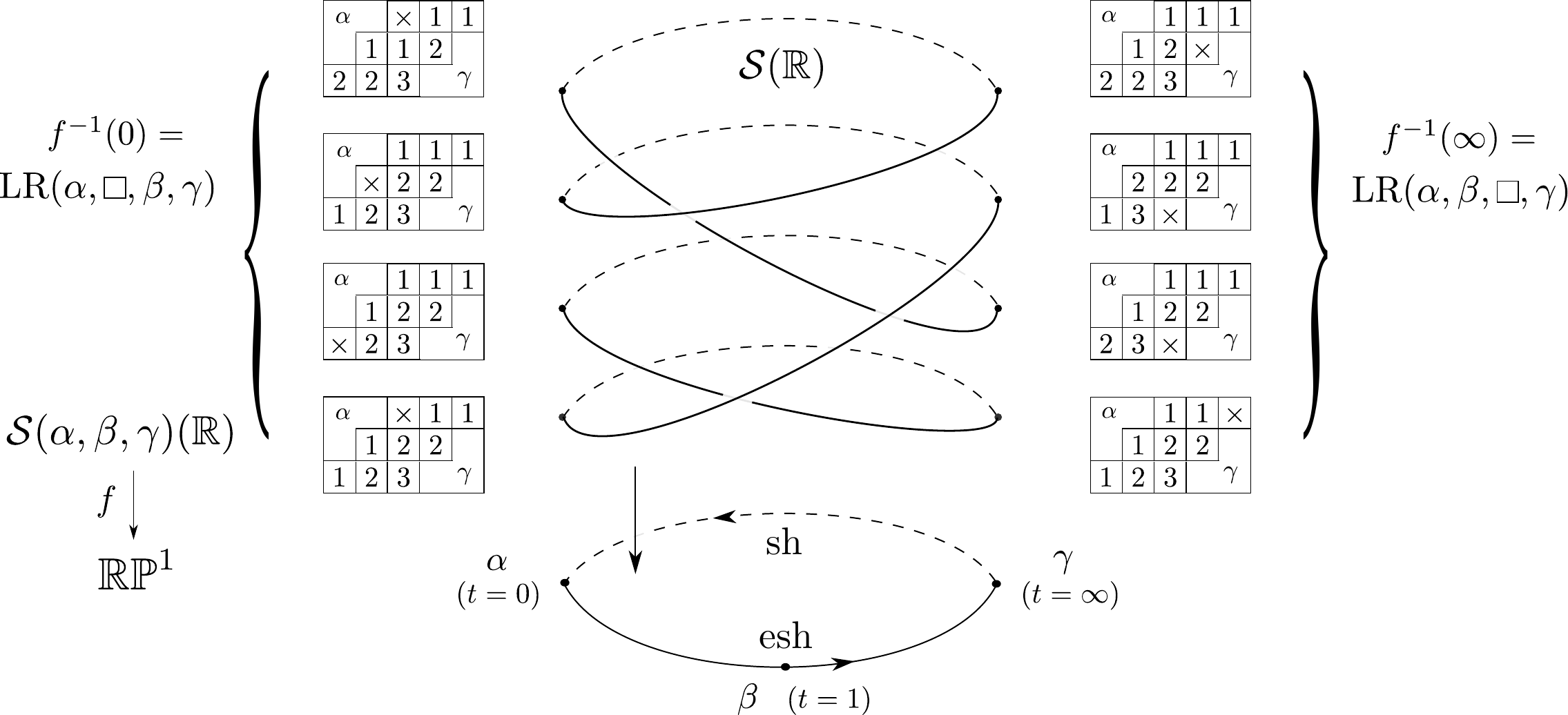}
%\vspace{-0.2cm}
\caption{An example of the covering space of Theorem \ref{thm:intro-2}.  The fibers over $0$ and $\infty$ are indexed by chains of tableaux, with $\ybox$ denoting the single box.  The dashed arcs correspond to sliding the $\ybox$ through the tableau using jeu de taquin.  The monodromy operator is $\omega = \sh \circ \esh$.}
\label{fig:covering-space}
\end{figure}

%\vspace{-0.25cm}

The operators $\esh$ and $\omega$ are our objects of study.  In \cite{bib:Levinson}, the second author described $\esh$ as the conjugation of jeu de taquin \emph{promotion} by \emph{rectification} (see Section \ref{sec:background} for a precise definition).  Variants of this operation have appeared elsewhere in \cite{bib:BerKir}, \cite{bib:HenrKam}, \cite{bib:KirBer}.

%Note that $\omega$ is not a commutator in the sense of group theory, since rectification and promotion are maps between different sets of tableaux. Indeed, $\omega$ may be an odd permutation; computing the sign of $\omega$ combinatorially and relating it to the $K$-theory of the Grassmannian is one of our main applications.  We return to this consideration below.

We prove two main theorems. The first is a shorter, `local' combinatorial description of the map $\esh$, which no longer requires rectifying or otherwise modifying the skew shape. We call our algorithm \emph{local evacuation shuffling}. Local evacuation-shuffling resembles jeu de taquin: it consists of successively moving the $\ybox$ through $T$ through a weakly increasing sequence of squares. Unlike jeu de taquin, the path is in general disconnected. (See Section \ref{sec:local-esh} for the definition, and Figure \ref{fig:antidiagonal} for a visual description of the path of the $\ybox$.)

\begin{figure}[b]
\centering
\includegraphics[scale=0.8]{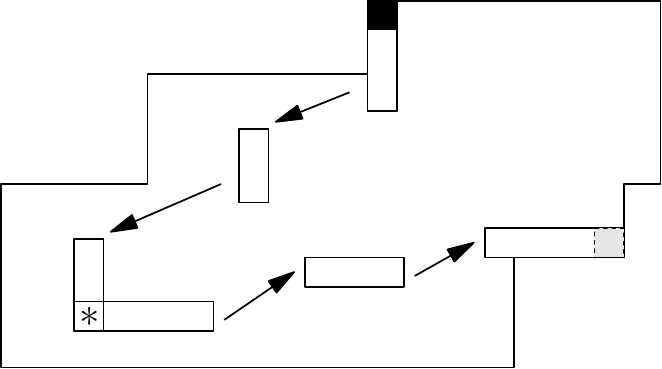}
\caption{The path of the $\ybox$ in a local evacuation-shuffle.  The black and gray squares are the initial and final locations of the $\ybox$; the algorithm switched from ``Phase 1'' to ``Phase 2'' at the square marked by a $*$. There is an \emph{antidiagonal symmetry}: the Phase 1 path forms a vertical strip, while the Phase 2 path forms a horizontal strip. We characterize this symmetry precisely in Corollary \ref{cor:antidiag-evacu-path}.}
\label{fig:antidiagonal}
\end{figure}

\begin{theorem}\label{thm:MainResult1}
  The map $\esh$ agrees with local evacuation shuffling. In particular, $\omega = \sh \circ \lesh$.
\end{theorem}

Our second main result is related to K-theory and the orbit structure of $\omega$. We first recall a key consequence of Theorem \ref{thm:intro-2}:

\begin{proposition}[\cite{bib:Levinson}, Lemma 5.6] \label{prop:numerics}
Let $S$ have $\iota(S)$ irreducible components and let $S(\mathbb{R})$ have $\eta(S)$ connected components. Let $\chi(\mathcal{O}_S)$ be the holomorphic Euler characteristic. Then
\begin{align*}
\eta(S) &\geq \iota(S) \geq \chi(\mathcal{O}_S) \text{ and } \\
\eta(S) &\equiv \chi(O_S) \pmod 2.
\end{align*}
\end{proposition}

We note that $\eta(S)$ is the number of orbits of $\omega$, viewed as a permutation of $\LRyb$. The numerical consequences above are most interesting in the context of K-theoretic Schubert calculus, which expresses $\chi(\mathcal{O}_S)$ in terms of both ordinary and K-theoretic \newword{genomic tableaux}, namely
\[\chi(\mathcal{O}_S) = |\LRyb| - |\Kabg|.\]
See Section \ref{sec:K-theory} for the definition of $K(\gamma^c/\alpha; \beta)$ due to Pechenik-Yong \cite{bib:Pechenik}. In particular, we see that
\begin{align}\label{eqn:ktheory-ineq-A}
|\Kabg| &\geq |\LRby| - |\mathrm{orbits}(\omega)|, \text{ and} \\
\label{eqn:ktheory-mod2-A}
|\Kabg| &\equiv |\LRby| - |\mathrm{orbits}(\omega)| \pmod 2.
\end{align}
The following reformulation is instructive: we recall that the \emph{reflection length} of a permutation $\sigma \in S_N$ is the minimum length of a factorization of $\sigma$ into arbitrary (not necessarily adjacent) transpositions. We have
\[\mathrm{rlength}(\sigma) = \sum_{\mathcal{O} \in \mathrm{orbits}(\sigma)}(|\mathcal{O}| - 1) = N - |\mathrm{orbits}(\sigma)|.\]
We also recall that the \textit{sign} of a permutation is the parity of the reflection length:
\[ \mathrm{sgn}(\sigma) \equiv \mathrm{rlength}(\sigma)\pmod 2.\]
where we use the convention that the sign of a permutation is $0$ or $1$ (rather than $\pm 1$).  Applying these relations to equations \eqref{eqn:ktheory-ineq-A} and \eqref{eqn:ktheory-mod2-A}, we see that
\begin{align} \label{eqn:ktheory-ineq}
|\Kabg| &\geq \mathrm{rlength}(\omega), \text{ and} \\
\label{eqn:ktheory-mod2}
|\Kabg| &\equiv \mathrm{sgn}(\omega) \pmod 2.
\end{align}

For the case where $\beta$ is a horizontal strip, a combinatorial interpretation of these facts was given in \cite{bib:Levinson}, indexing all but one step of an orbit by genomic tableaux. Our second main result generalizes this combinatorial interpretation, showing that certain steps of local evacuation-shuffling correspond bijectively to the genomic tableaux $\Kabg$:

\begin{theorem}\label{thm:MainResult2}
As $T$ ranges over $\LRyb$, for either phase of the local description of $\esh(T)$, the gaps in the $\ybox$ path are in bijection with the set $K(\gamma^c/\alpha;\beta)$.
\end{theorem}
Using the bijections of Theorem \ref{thm:MainResult2}, we give an independent, purely combinatorial proof of the relations \eqref{eqn:ktheory-ineq} and \eqref{eqn:ktheory-mod2}, by factoring $\omega$ into auxiliary operators $\omega_i$, which roughly correspond to the individual steps of local evacuation-shuffling, applied in isolation. If $\beta$ has $\ell(\beta)$ parts, we have the following:

\begin{theorem}\label{thm:intro-parity} 
There is a factorization of $\omega$ as a composition $\omega_{\ell(\beta)} \cdots \omega_{1}$, such that for every $i$, and every orbit $\mathcal{O}$ of $\omega_i$, the bijections of Theorem \ref{thm:MainResult2} yield \emph{exactly} $|\mathcal{O}|-1$ distinct genomic tableaux.
\end{theorem}
\noindent By summing over the orbits of the $\omega_i$'s, we deduce
\begin{align*}
\mathrm{rlength}(\omega) \leq \sum_i \mathrm{rlength}(\omega_i) = \sum_{i,\mathcal{O}} (|\mathcal{O}|-1) =|\Kabg|,
\end{align*}
by the subadditivity of reflection length. The sign computation is analogous.

Finally, we conjecture that the inequality \eqref{eqn:ktheory-ineq} `applies orbit-by-orbit on $\omega$', in the following sense:
\begin{conjecture} \label{conj:orbit-by-orbit-intro}
Using the bijections of Theorem \ref{thm:MainResult2}, each orbit $\mathcal{O}$ of $\omega$ generates \emph{at least} $|\mathcal{O}| - 1$ genomic tableaux.
\end{conjecture}
Conjecture \ref{conj:orbit-by-orbit-intro} implies the inequality \eqref{eqn:ktheory-ineq}, by summing over the orbits of $\omega$. In section \ref{sec:omega-orbits}, we prove this conjecture in certain special cases and give computational evidence that it holds in general.

The paper is organized as follows.  In Section \ref{sec:background}, we briefly recall the necessary background and definitions from the theory of tableaux and dual equivalence.  In Section \ref{sec:local-esh}, we define $\lesh$ and establish its basic properties. In Section \ref{sec:main-result}, we prove Theorem \ref{thm:MainResult1} that $\lesh$ agrees with $\esh$. We also establish certain symmetries of the algorithm under rotation and transposition of the tableau. Section \ref{sec:K-theory} contains the link to K-theory, and the proofs of Theorems \ref{thm:MainResult2} and \ref{thm:intro-parity}.

The remaining sections explore some consequences of the main results, including new geometric facts about Schubert curves. Section \ref{sec:omega-orbits} contains the results on orbits of $\omega$, including a characterization of its fixed points. In Section \ref{sec:constructions}, we construct examples of Schubert curves with `extremal' geometric properties. In Section \ref{sec:conjectures} we state some remaining combinatorial and geometric conjectures.

\subsection{Acknowledgments}
We especially thank Oliver Pechenik for his help with testing our conjectures using Sage \cite{sage}, and for several helpful discussions about tableaux combinatorics. Computations in Sage \cite{sage} were very helpful for testing conjectures and verifying our results throughout this work. We also thank Mark Haiman and David Speyer for their guidance. Finally, we are grateful to Bryan Gillespie, Nic Ford, Gabriel Frieden, Rachel Karpman, Greg Muller and David Speyer for comments on earlier drafts of this paper.

\section{Background and Notation}\label{sec:background}

%%%%%%%%%%%%%%%%%%%%%%%%%
%%%%% BACKGROUND    %%%%%
%%%%%%%%%%%%%%%%%%%%%%%%%

\subsection{Partitions and tableaux}
Let $\lambda=(\lambda_1\ge \cdots \ge \lambda_k)$ be a partition.  We will refer to the partition $\lambda$ and its \newword{Young diagram} interchangeably throughout, where we use the English convention for Young diagrams in which there are $\lambda_i$ squares placed in the $i$th row from the top. The \newword{corners} of $\lambda$ are the squares which, if removed, leave a smaller partition behind. The \newword{co-corners} are, dually, the exterior squares which, if added to $\lambda$, give a larger partition. The \newword{transpose} of $\lambda$ is the partition $\lambda^\ast$ obtained by transposing its Young diagram. The \newword{length} of $\lambda$ is $\ell(\lambda) = k$. 

If $\mu=(\mu_1\ge \cdots \ge \mu_r)$ is a partition with $r\le k$ and $\mu_i\le \lambda_i$ for all $i$, then the \newword{skew shape} $\lambda/\mu$ is the diagram formed by deleting the squares of $\mu$ from that of $\lambda$.  The \newword{size} of $\lambda/\mu$, denoted $|\lambda/\mu|$, is the number of squares that remain in the diagram.

We will occasionally refer to (co-)corners of a skew shape $\lambda/\mu$. The \newword{inner} (respectively, \newword{outer}) \newword{corners} of $\lambda/\mu$ are the corners of $\lambda$ (respectively, the co-corners of $\mu$). These are the squares which, if deleted, leave a smaller skew shape. Similarly, the \newword{inner} (resp. \newword{outer}) \newword{co-corners} are the co-corners of $\lambda$ (resp. the corners of $\mu$): the exterior squares which can be added to obtain a larger skew shape (Figure \ref{fig:corners}).\footnote{Note that the definition of corner of $\lambda/\mu$ depends on the pair of partitions $\lambda$ and $\mu$, not just the squares that make up the skew shape. The same square may be both an inner and outer corner; likewise for co-corners.}

\begin{figure}[ht]
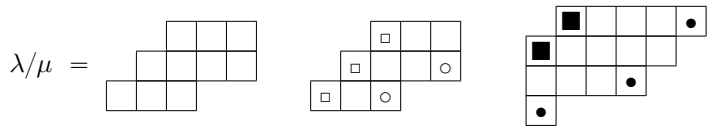

\centering
\[\lambda/\mu \ =\ \young(::\hfil\hfil\hfil,:\hfil\hfil\hfil\hfil,\hfil\hfil\hfil) \qquad \young(::\square\hfil\hfil,:\square\hfil\hfil\circ,\square\hfil\circ) \hspace{0.85cm} \raisebox{-7mm}{\young(:\blacksquare\hfil\hfil\hfil\bullet,\blacksquare\hfil\hfil\hfil\hfil,\hfil\hfil\hfil\bullet,\bullet)}\]
\caption{Corners ($\square,\circ$) and co-corners ($\blacksquare,\bullet$).}
\label{fig:corners}
\end{figure}

We write $\rect=((n-k)^k)$ to denote a fixed rectangular shape of size $k\times (n-k)$, and we will always work with skew shapes that fit inside $\rect$.  We define the \newword{complement} of a partition $\lambda\subset \rect$, denoted $\lambda^c$, to be the partition $(n-k-\lambda_k,n-k-\lambda_{k-1},\ldots,n-k-\lambda_1)$.  Note that $\lambda^c$ can be formed by rotating $\lambda$ by $180^\circ$ about the center of $\rect$ and then removing it from $\rect$.

A \newword{semistandard Young tableau} (SSYT) of skew shape $\lambda/\mu$ is a filling of the boxes of the Young diagram of $\lambda/\mu$ with positive integers such that the entries in are weakly increasing to the right across each row and strictly increasing down each column.   The \newword{content} of a semistandard Young tableau is the tuple $\beta=(\beta_1,\ldots,\beta_t)$ where $\beta_i$ is the number of times the number $i$ appears in the filling.  The \newword{reading word} is the sequence formed by reading the rows from bottom to top, and left to right within a row.

\begin{figure}[ht]
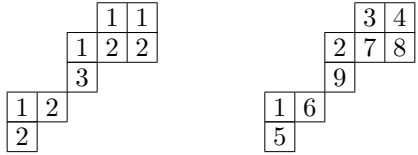

\centering
$$\young(:::11,::122,::3,12,2) \qquad \qquad \young(:::34,::278,::9,16,5)$$
\caption{Left: An SSYT with content $(2,2,1)$ and reading word $212312211$. Right: Its standardization.}
\label{fig:reading-word}
\end{figure}

An SSYT $S$ is \newword{standard} if the numbers $1,\ldots,|S|$ each appear exactly once as entries in $S$.  The \newword{standardization} of an SSYT $T$ is the tableau formed by replacing the entries of $T$ with the numbers $1,\ldots,|T|$ in the unique way that preserves the relative ordering of the entries, where ties are broken according to left-to-right ordering in the reading word (Figure \ref{fig:reading-word}).

The \newword{suffix} of an entry $m$ of $T$ is the suffix of the reading word consisting of the letters \emph{strictly} after $m$.  The \newword{weak suffix} is the suffix including that letter and those after it.  A suffix is \newword{ballot for $(i,i+1)$} if it contains at least as many $i$'s as $i+1$'s, and is \newword{tied} if it has the same number of $i$'s as $i+1$'s. Finally, a semistandard Young tableau $T$ is \newword{ballot} or \newword{Littlewood-Richardson} (also known as \emph{Yamanouchi} or \emph{lattice}) if every weak suffix of its reading word is ballot for $(i,i+1)$, for all $i$.  We write $\LR_\mu^\lambda(\beta)$ for the set of (semistandard) Littlewood-Richardson tableaux of shape $\lambda/\mu$ and content $\beta$. 

A tableau of shape $\lambda/\mu$ is \newword{straight shape}, or \newword{shape $\lambda$}, if $\mu=\eset$ is the empty partition.  The \newword{highest weight tableau} of straight shape $\lambda$ is the tableau in which the $i$th row from the top is filled with all $i$'s.  It is easily verified that this tableau is the only Littlewood-Richardson tableau of straight shape $\lambda$.

\subsubsection{Jeu de taquin rectification and shuffling}

An \newword{inward (resp. outward)  jeu de taquin} slide of a semistandard skew tableau $T$ is the operation of starting with an inner (resp. outer) co-corner of $T$ as the \newword{empty square}, and at each step sliding either the entry below or to the right (resp. above or to the left) of the empty square into that square in such a way that the resulting tableau is still semistandard. This condition uniquely determines the choice of slide. The former position of the moved entry is the new empty square, and the process continues until the empty square is an outer (resp. inner) co-corner of the remaining tableau.  An example is shown below.
$$\young(:13,23)\hspace{0.2cm} \longrightarrow \hspace{0.2cm}\young(1\cdot 3,23) \hspace{0.2cm}\longrightarrow\hspace{0.2cm} \young(133,2)$$
See \cite{bib:Fulton} for a more detailed introduction to jeu de taquin.

The \newword{rectification} of a skew tableau $T$, denoted $\rectify(T)$, is defined to be the straight shape tableau formed by any sequence of inwards jeu de taquin slides.  It is well known (often called the ``fundamental theorem of jeu de taquin'') that any sequence of slides results in the same rectified tableau.  

\begin{definition}
Let $S,T$ be semistandard skew tableaux so that the shape of $T$ extends the shape of $S$, that is, $T$ can be formed by successively adding outer co-corners starting from $S$.  We define the (jeu de taquin) \newword{shuffle} of $(S,T)$ to be the pair of tableaux $(T',S')$, where $S'$ is obtained from $S$ by performing outwards jeu de taquin slides in the order specified by the standardization of $T$, and $T'$ is obtained from $T$ by performing reverse slides in the order specified by the standardization of $S$.
\end{definition}

Equivalently, $T'$ records the squares vacated by $S$ as $S$ slides outwards, and $S'$ records the squares vacated by $T$ as $T$ slides inwards.  We then say $S$ and $S'$ are \newword{slide equivalent}, and likewise for $T, T'$. 

\begin{lemma}
Shuffling is an involution.
\end{lemma}
\begin{proof}[Proof sketch]
Shuffling can computed by growth diagrams (see \cite{bib:StanleyEC2}, appendix A1.2), with the input on the left and top sides, and the output on the bottom and right sides. The transpose of a growth diagram is again a growth diagram. \qed
\end{proof}

\subsection{Dual equivalence} 

We will use the theory of dual equivalence, particularly Lemmas \ref{lem:upper-shuffle} and \ref{lem:outer-esh}, to prove Theorem \ref{thm:MainResult1} on the correctness of our local algorithm for the monodromy operator $\omega$. Dual equivalence is not used outside of Section \ref{sec:main-result}.

 Let $S,S'$ be skew standard tableaux of the same shape. Following the conventions of \cite{bib:Haiman}, we say $S$ is \newword{dual equivalent} to $S'$ if the following is always true: let $T$ be a skew standard tableau whose shape extends, or is extended by, that of $S$. Let $\widetilde{T}, \widetilde{T}'$ be the results of shuffling $T$ with $S$ and with $S'$. Then $\widetilde{T} = \widetilde{T}'$.

In other words, $S$ and $S'$ are dual equivalent if they have the same shape, and they transform \emph{other} tableaux the same way under jeu de taquin.  Therefore, the fact that rectification of skew tableaux is well-defined, regardless of the rectification order, can be phrased in terms of dual equivalence as follows.
\begin{theorem} \label{thm-dual-jdt}
Any two tableaux of the same straight shape are dual equivalent.
\end{theorem}
\begin{definition}
We will write $D_\beta$ for the unique dual equivalence class of straight shape $\beta$.
\end{definition}

It is also known \cite{bib:Haiman} that $S$ and $S'$ are dual equivalent if \emph{their own} shapes evolve the same way under any sequence of slides.  We state this in the following lemma.

\begin{lemma} \label{lem-dual-def2}
Let $S,S'$ be skew standard tableaux of the same shape. Then $S$ is dual equivalent to $S'$ if and only if the following is always true: 
\begin{itemize}
\item Let $T$ be a tableau whose shape extends, or is extended by, that of $S$. Let $\widetilde{S}$ and $\widetilde{S'}$ be the results of shuffling $S,S'$ with $T$. Then $\widetilde{S}$ and $\widetilde{S}'$ have the same shape.
\end{itemize}
Additionally, in this case $\widetilde{S}$ and $\widetilde{S}'$ are also dual equivalent.
\end{lemma}

We can extend the definition of shuffling to dual equivalence classes, using the following result. \cite{bib:Haiman}

\begin{lemma}
Let $S,T$ be skew tableaux, with $T$'s shape extending that of $S$, and let $(S,T)$ shuffle to $(\widetilde{T},\widetilde{S})$. The dual equivalence classes of $\widetilde{T}$ and $\widetilde{S}$ depend only on the dual equivalence classes of $S$ and $T$.
\end{lemma}

So we may use any tableau of straight shape $\mu$ to rectify a skew tableau $S$ of shape $\lambda/\mu$. Thus we may speak of the \newword{rectification tableau} of a slide equivalence class. Similarly, by the above facts, we may speak of the \newword{rectification shape of a dual equivalence class} $\rsh(D)$.  This is the shape of any rectification of any representative of the class $D$.

\newcommand{\cD}{\mathcal{D}}
\newcommand{\cS}{\mathcal{S}}
\begin{lemma}\label{slide-dual}
Let $\cD,\cS$ be a dual equivalence class and a slide equivalence class, with $\rsh(\cD) = \rsh(\cS)$. There is a unique tableau in $\mathcal{D} \cap \mathcal{S}$.
\end{lemma}
\begin{proof}
Uniqueness is clear. To produce the tableau, pick any $T_\cD \in \cD$. Rectify $T_\cD$ using an arbitrary tableau $X$, so $(X,T_\cD)$ shuffles to $(\widetilde{T_\cD},\widetilde{X})$ (and $X$ and $ \widetilde{T_\cD}$ are of straight shape). Replace $\widetilde{T_\cD}$ by the rectification tableau $R_\cS$ for the class $\cS$, and let $(R_\cS,\widetilde{X})$ shuffle back to $(X,T)$. Then $T$ and $R_\cS$ are slide equivalent, and by Theorem \ref{thm-dual-jdt} and Lemma \ref{lem-dual-def2}, $T$ and $T_\cD$ are dual equivalent. \qed
\end{proof}

The dual equivalence classes of a given skew shape and rectification shape are counted by a Littlewood-Richardson coefficient.

\begin{lemdef} \label{lem:dual-LRcoeff}
Let $\lambda/\mu$ be a skew shape and let 
\[\DE_\mu^\lambda(\beta) = \{\text{dual equivalence classes } D \text{ with } \sh(D) = \lambda/\mu \text{ and } \rsh(D) = \beta \}.\]
Then $|\DE_\mu^\lambda(\beta)| = c_{\mu \beta}^\lambda.$
\end{lemdef}
\begin{proof}
It is well-known that $c_{\mu \beta}^\lambda$ counts tableaux $T$ of shape $\lambda/\mu$ whose rectification is the standardization of the highest-weight tableau of shape $\beta$. This specifies the slide equivalence class of $T$; by Lemma \ref{slide-dual}, such tableaux are in bijection with $\DE_\mu^\lambda(\beta)$. \qed
\end{proof}

\subsubsection{Connection to Littlewood-Richardson tableaux}

%As noted in the proof of Lemma \ref{lem:dual-LRcoeff}, we know by Lemma \ref{slide-dual} that a dual equivalence class $D$ of rectification shape $\beta$ has a unique \newword{highest-weight representative}, defined as the unique tableau $T$ with dual equivalence class $D$ and slide equivalence class given by the highest weight tableau of shape $\beta$.   By the fundamental theorem of jeu de taquin, if $S,T$ are highest-weight skew tableaux, the shuffles $(T',S')$ are also of highest weight.  Consequently, we can think of highest-weight representatives as the `canonical' representatives of their dual equivalence class, and work directly with the tableaux.  

As noted in the proof of Lemma \ref{lem:dual-LRcoeff}, we know by Lemma \ref{slide-dual} that a dual equivalence class $D$ of rectification shape $\beta$ has a unique \newword{highest-weight representative}, that is, the unique tableau $T$ dual equivalent to $D$ and slide equivalent to the standardization of the highest weight tableau of shape $\beta$.  By the fundamental theorem of jeu de taquin, if $S,T$ are highest-weight skew tableaux, the shuffles $(T',S')$ are also of highest weight.  We wish to work with Littlewood-Richardson tableaux, which are in bijection with these highest-weight representatives:

%\begin{lemma}\label{lem:DE-LR}
%  A semistandard skew tableau $T$ is the highest weight representative of its dual equivalence class $D$ if and only if its reading word is ballot.  Moreover, $\rsh(D)=\beta$ if and only if $T$ has content $\beta$. 
%\end{lemma}

\begin{lemma}\label{lem:DE-LR}
  A semistandard skew tableau $T$ is Littlewood-Richardson (of content $\beta$) if and only if its standardization is the highest weight representative of its dual equivalence class $D$ (and $\rsh(D)=\beta$).
\end{lemma}

  This is well-known; see e.g. \cite{bib:Fulton}.   A consequence of this lemma is that there is a canonical bijection $$\DE_\mu^\lambda(\beta)\cong \LR_\mu^\lambda(\beta).$$  If $T$ is a highest weight representative for $D$ and $\beta$ is understood, we often write $$\HW(D)=T \text{ and } \DE(T)=D.$$

\subsubsection{Transposing and rotating dual equivalence classes}\label{sec:rotate-transpose-DE}

Let $T$ be a standard tableau of skew shape $\alpha/\beta$, and write $T^R$ for the tableau of shape $\beta^c/\alpha^c$ obtained by rotating $T$ by $180^\circ$, then reversing the numbering of its entries. Rotating commutes with jeu de taquin shuffling, so the dual equivalence class of $T^R$ depends only on the dual equivalence class of $T$. This gives an involution of dual equivalence classes
\[D \mapsto D^R : \DE_\mu^\lambda(\beta) \to \DE^{\mu^c}_{\lambda^c}(\beta).\]
In particular, any tableaux $T, T'$ of `anti-straight-shape' $\rect/\lambda^c$ are dual equivalent, and their rectifications have shape $\lambda$. The same remarks apply to transposing standard tableaux, so we may speak of transposing a dual equivalence class:
\[D \mapsto D^\ast : \DE_\mu^\lambda(\beta) \to \DE_{\mu^\ast}^{\lambda^\ast}(\beta^\ast).\]
We note that these operations do not correspond to simple operations on the Littlewood-Richardson tableau $LR(D)$. The combination, however, is straightforward:\footnote{This phenomenon reflects the fact that both transformations encode the `Fundamental Symmetry' of Young tableau bijections, in the sense of Pak and Vallejo's work in \cite{bib:PakVallejo}.  Consequently, the composition \emph{does not} encode this deep symmetry, hence is easier to compute.}

\begin{lemma} \label{lem:highwt-lowwt}
Let $D \in \DE_\mu^\lambda(\beta)$. Let $\tilde{D} = (D^R)^\ast$ be obtained by rotating \emph{and} transposing $D$.

Then $\tilde{T} = \LR(\tilde{D})$ is obtained from $T = \LR(D)$ as follows: for each $j = 1, \ldots, \beta_1$, let $V_j$ be the vertical strip containing the $j$-th-from-last instance of each entry $i$ in $T$. The squares obtained by rotating and transposing $V_j$ contain the entry $j$ in $\tilde{T}$.
\end{lemma}
\begin{proof}\
We defer the proof to Section \ref{sec:main-result}, where we prove a stronger statement (Lemma \ref{lem:5-facts}). \qed
\end{proof}

\subsection{Chains of dual equivalence classes and tableaux}

Following the conventions of \cite{bib:Levinson}, we define a \newword{chain of dual equivalence classes} to be a sequence $(D_1, \ldots, D_r)$ of dual equivalence classes, such that the shape of $D_{i+1}$ extends that of $D_i$ for each $i$ (Figure \ref{fig:DualChain}). We say the chain has \newword{type} $(\lambda^{(1)}, \ldots, \lambda^{(r)})$ if for each $i$, $\rsh(D_i) = \lambda_i$.

\begin{lemdef} \label{lem:dual-multiLRcoeff}
Let $\DE_\mu^\nu(\lambda^{(1)}, \ldots, \lambda^{(r)})$ denote the set of chains of dual equivalence classes of type $(\lambda^{(1)}, \ldots, \lambda^{(r)})$, such that $D_1$'s shape extends $\mu$ and $\nu$ is the outer shape of $D_r$. This has cardinality equal to the Littlewood-Richardson coefficient $c_{\mu, \lambda^{(1)}, \ldots, \lambda^{(r)}}^\nu$.
\end{lemdef}

\begin{figure}[t]
\begin{center}
  \includegraphics[scale=0.85]{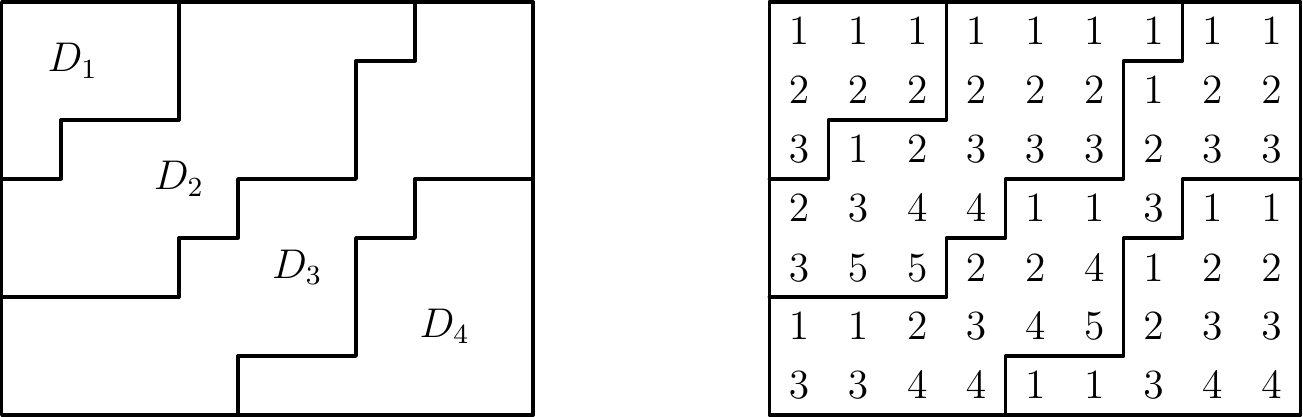}
\end{center}
\caption{\label{fig:DualChain} At left, a chain of dual equivalence classes that extend each other to fill a $k\times (n-k)$ rectangle, with rectification shapes $\lambda^{(1)},\ldots,\lambda^{(4)}$.  At right, a Littlewood-Richardson tableau with content $\lambda^{(i)}$ is given for the $i$th skew shape for $i=1,\ldots,4$.  Each dual equivalence class $D_\lambda$ of skew shape $\nu/\mu$ is represented by a unique Littlewood-Richardson tableau.}
\end{figure}

By Lemma \ref{lem:DE-LR}, we can work with Littlewood-Richardson tableaux in place of dual equivalence classes.  Define a \newword{chain of Littlewood-Richardson tableaux} to be a sequence $(T_1, \ldots, T_r)$ of Littlewood-Richardson tableaux, such that the shape of $T_{i+1}$ extends that of $T_i$ for each $i$.  We say the chain has \newword{type} $(\lambda^{(1)}, \ldots, \lambda^{(r)})$ if $T_i$ has content $\lambda^{(i)}$ for each $i$.

\begin{lemdef} \label{lem:LRchain}
Let $\LR_\mu^\nu(\lambda^{(1)}, \ldots, \lambda^{(r)})$ denote the set of chains of Littlewood-Richardson tableaux of type $(\lambda^{(1)}, \ldots, \lambda^{(r)})$, such that $T_1$'s shape extends $\mu$, and $\nu$ is the outer shape of $T_r$. There is a natural bijection $$\LR_\mu^\nu(\lambda^{(1)}, \ldots, \lambda^{(r)})\cong \DE_\mu^\nu(\lambda^{(1)},\ldots,\lambda^{(r)}).$$
\end{lemdef}

\begin{definition}
 When $\mu=\eset$ and $\nu=\rect$, we simply write $\LR(\lambda^{(1)},\ldots,\lambda^{(r)})$ and $\DE(\lambda^{(1)},\ldots,\lambda^{(r)})$ in place of $\LR_\mu^\nu(\lambda^{(1)},\ldots,\lambda^{(r)})$ and $\DE_\mu^\nu(\lambda^{(1)},\ldots,\lambda^{(r)})$ respectively.
\end{definition}

\subsubsection{Operations on chains} 

\label{sec:shuffling-ops} We define the \newword{shuffling} operations
\[\sh_i : \DE_\mu^\nu(\lambda^{(1)}, \ldots, \lambda^{(i)}, \lambda^{(i+1)}, \cdots \lambda^{(r)}) \to \DE_\mu^\nu(\lambda^{(1)}, \ldots, \lambda^{(i+1)}, \lambda^{(i)}, \cdots \lambda^{(r)})\]
\[\sh_i : \LR_\mu^\nu(\lambda^{(1)}, \ldots, \lambda^{(i)}, \lambda^{(i+1)}, \cdots \lambda^{(r)}) \to \LR_\mu^\nu(\lambda^{(1)}, \ldots, \lambda^{(i+1)}, \lambda^{(i)}, \cdots \lambda^{(r)})\]
by shuffling $(D_i,D_{i+1})$ or $(T_i,T_{i+1})$ respectively.  The shuffling operations commute with the correspondence between $\DE$ and $\LR$ of Lemma \ref{lem:LRchain}.  They satisfy the relations $\sh_i^2 = \mathrm{id}$ and $\sh_i \sh_j = \sh_j \sh_i$ when $|i-j| > 1$. Note, however, that $\sh_i \sh_{i+1} \sh_i \ne \sh_{i+1} \sh_i \sh_{i+1}$ in general.

We next define the $i$-th \newword{evacuation} operations 
\[\ev_i : \DE_\mu^\nu(\lambda^{(1)}, \ldots, \lambda^{(r)}) \to \DE_\alpha^\beta(\lambda^{(i)}, \ldots, \lambda^{(1)}, \lambda^{(i+1)}, \ldots, \lambda^{(r)})\]
\[\ev_i : \LR_\mu^\nu(\lambda^{(1)}, \ldots, \lambda^{(r)}) \to \LR_\alpha^\beta(\lambda^{(i)}, \ldots, \lambda^{(1)}, \lambda^{(i+1)}, \ldots, \lambda^{(r)})\]
by $\ev_i =  \sh_1 (\sh_2 \sh_1) \cdots (\sh_{i-2} \cdots \sh_1) (\sh_{i-1} \cdots \sh_1)$. This results in reversing the first $i$ parts of the chain's type, by first shuffling $D_1$ (or $T_1$) outwards past $D_i$, then shuffling the $D_2'$ (now the first element of the chain) out past $D_i'$, and so on.

In the case where $\mu = \eset$ and $\lambda^{(i)} = \ebox$ for all $i$, the operation $\ev_i$ reduces to evacuation of the standard tableau formed by the first $i$ entries. In general, $\ev_i$ is an involution:
\begin{lemma} \label{evac-involution}
The operation $\ev_i$ is an involution.
\end{lemma}
\begin{proof}
By definition, $\ev_i = \ev_{i-1} (\sh_{i-1} \cdots \sh_1)$. On the other hand, observe that $(\sh_{i-1} \cdots \sh_1)\ev_i = \ev_{i-1}$. (Each extra $\sh_j$ cancels the leftmost instance of $\sh_j$ in $\ev_i$.) Thus we have
\[\ev_i^2 = \ev_{i-1}(\sh_{i-1} \cdots \sh_1) \ev_i = \ev_{i-1}^2,\]
and the claim follows by induction. \qed
\end{proof}

Finally, we define the $i$-th \newword{evacuation-shuffle} operations
\[\esh_i : \DE_\eset^{\rect}(\lambda^{(1)}, \ldots, \lambda^{(i)}, \lambda^{(i+1)}, \cdots \lambda^{(r)}) \to \DE_\eset^{\rect}(\lambda^{(1)}, \ldots, \lambda^{(i+1)}, \lambda^{(i)}, \cdots \lambda^{(r)})\]
\[\esh_i : \LR_\eset^{\rect}(\lambda^{(1)}, \ldots, \lambda^{(i)}, \lambda^{(i+1)}, \cdots \lambda^{(r)}) \to \LR_\eset^{\rect}(\lambda^{(1)}, \ldots, \lambda^{(i+1)}, \lambda^{(i)}, \cdots \lambda^{(r)})\]
by
\[\esh_i = \ev_{i+1}^{-1} \sh_1 \ev_{i+1}.\]
This operation is simpler than it appears: it only affects the $i$-th and $(i+1)$-th entries of the chain, and its effect is local (it depends only on the $i$-th and $(i+1)$-th entries). We have the following:
\begin{lemma}[\cite{bib:Levinson}, Lemma 3.12] \label{lem:upper-shuffle}
Let ${\bf D} = (D_1,\ldots, D_r) \in \DE_\eset^{\rect}(\lambda^{(1)}, \ldots, \lambda^{(r)})$ and write
\[\esh_i({\bf D}) = (D_1', \ldots, D'_{i+1}, D'_i, \ldots, D_r').\]
Then:
\begin{itemize}
\item[(i)] We have $D_j = D_j'$ for all $j \ne i,i+1$.
\item[(ii)] The remaining two classes $D_i', D_{i+1}'$ are computed as follows: let $D_1 \sqcup \cdots \sqcup D_{i-1} = D_\tau$ be the concatenation of the first $i-1$ classes (i.e. the unique class of straight-shape $\tau$, the outer shape of $D_{i-1}$). Let $\sigma$ be the outer shape of $D_{i+1}$. Consider $\overline{\bf D} = (D_\tau, D_{i}, D_{i+1}) \in \DE_\eset^\sigma(\tau, \lambda^{(i)}, \lambda^{(i+1)})$. Then
\[\esh_2(\overline{\bf D}) = \sh_1 \sh_2 \circ \sh_1 \circ \sh_2 \sh_1(\overline{\bf D}) = (D_\tau, D_{i+1}', D_i').\]
\end{itemize}
\end{lemma}
\noindent In other words, evacuation-shuffling a pair of consecutive tableaux $(S,T)$ in a Littlewood-Richardson chain consists of rectifying $(S,T)$ together, then shuffling them, then un-rectifying.

We may also compute $\esh_i$ by anti-rectifying into the lower right corner of the rectangle instead of rectifying:

\begin{lemma} \label{lem:outer-esh}
Let ${\bf D} = (D_1, D_2, D_3, D_4) \in \DE_\eset^\rect(\lambda^{(1)}, \lambda^{(2)}, \lambda^{(3)}, \lambda^{(4)})$. Then
\[\esh_2({\bf D}) = \sh_3 \sh_2 \circ \sh_3 \circ \sh_2 \sh_3({\bf D}).\]
\end{lemma}
\begin{proof}
Rotating dual equivalence classes, as in Section \ref{sec:rotate-transpose-DE},
\[\mathrm{rev} : (D_1, D_2, D_3, D_4) \mapsto (D_4^R, D_3^R, D_2^R, D_1^R),\]
corresponds to the word
\[\mathrm{rev} = \sh_1 \circ \sh_2\sh_1 \circ \sh_3\sh_2\sh_1.\]
(See \cite{bib:Speyer} for a proof via dual equivalence growth diagrams.) We have
\[\mathrm{rev} \circ \sh_i = \sh_{4-i} \circ \mathrm{rev}, \ \ \text{and so}\ \ \mathrm{rev} \circ \esh_2 \circ \mathrm{rev} = \sh_3 \sh_2 \circ \sh_3 \circ \sh_2 \sh_3.\]
On the other hand, we see directly, by simplifying the corresponding words, that \[\mathrm{rev} \circ \esh_2 \circ \mathrm{rev} = \esh_2 \]  and the proof is complete. \qed 
\end{proof}

We remark that neither of Lemmas \ref{lem:upper-shuffle} or \ref{lem:outer-esh} is easy to prove directly for ballot semistandard tableaux. We will use them in the proof of Theorem \ref{thm:main-theorem}.

\subsection{The case of interest and the operator \texorpdfstring{$\omega$}{w}}
\label{sec:case-of-interest}

The geometry of Schubert curves (see Section \ref{sec:introduction}) suggests studying sets of the form
\[\DE_\eset^\rect(\lambda^{(1)}, \ebox, \lambda^{(2)}, \ldots, \lambda^{(r)}),\]
where $\rect$ is a $k\times(n-k)$ rectangle and $1 + \sum |\lambda^{(i)}| = k(n-k)$, with the composition of shuffles and evacu-shuffles
\[\omega = \sh_2 \circ \cdots \circ \sh_{r-1} \circ \esh_{r-1} \circ \cdots \circ \esh_2.\]
In general, $\omega$ describes the monodromy and real connected components of the Schubert curve
\[S(\lambda^{(1)}, \ldots, \lambda^{(r)}) = \mathrm{\Omega}(\lambda^{(1)}, \mathcal{F}_{t_1}) \cap \cdots \cap \mathrm{\Omega}(\lambda^{(r)}, \mathcal{F}_{t_r}),\]
where the osculation points $t_i$ are real numbers with $0 = t_1 < t_2 < \cdots < t_r = \infty.$  (See \cite{bib:Levinson}, Corollary 4.9.)
Our local description of $\esh$ will apply to each of the above $\esh_i$ operations, by Lemma \ref{lem:upper-shuffle}.  Therefore, our main results, in the case of three marked points, generalize without difficulty to this general case.  We leave these extensions to the interested reader.

Thus, for simplicity, we restrict for the remainder of the paper to the case of three partitions $\alpha, \beta, \gamma$, i.e. we study the operator
\[\omega = \sh_2 \circ \esh_2\]
on the sets
\[\DEyb \text{ and } \DEby,\]
or equivalently
\[\LRyb \text{ and } \LRby.\]

Since we mostly work only with $\sh_2$ and $\esh_2$, we often simply abbreviate them as $\sh$ and $\esh$, as in Section \ref{sec:introduction}.
% note: we do use the others in Section 4, when we factor classes into s-decompositions.

\begin{remark}[Notation] Since the straight shape $\alpha$ and anti straight shape $\gamma^c$ each have only one dual equivalence class, an element of $\DEyb$ can be thought of as a pair $(\ybox,D)$, with $D$ a dual equivalence class of rectification shape $\beta$, and $\ybox$ an inner co-corner of $D$, such that the shape of $\ybox \sqcup D$ is $\gamma^c/\alpha$. We represent elements of $\DEby$ similarly, with $\ybox$ as an outer co-corner. 

We will occasionally refer to the element as $D$ if the position of the $\ybox$ is understood. Similar remarks apply to $\LRyb$ and $\LRby$, and we write $(\ybox,T)$ or $(T,\ybox)$ (or simply $T$) to denote elements of these sets.
\end{remark}

\subsubsection{Connection to tableau promotion}

Combinatorially, $\omega$ can be thought of as a commutator of well-known operations on Young tableaux.  Computing $\esh(\ybox,T)$ is equivalent to the following steps:
\begin{itemize}
  \item \textbf{Rectification.} Treat the $\ybox$ as having value $0$ and being part of a semistandard tableau $\widetilde{T}=\ybox \sqcup T$.  Rectify, i.e. shuffle $(S,\widetilde{T})$ to $(\widetilde{T}',S')$, where $S$ is an arbitrary straight-shape tableau.
  \item \textbf{Promotion (see \cite{bib:StanleyEC2}).} Delete the $0$ of $\widetilde{T}'$ and rectify the remaining tableau.  Label the resulting empty outer corner with the number $\ell(\beta)+1$.
  \item \textbf{Un-rectification.} Un-rectify the new tableau by shuffling once more with $S'$.  Replace the $\ell(\beta)+1$ by $\ybox$.
\end{itemize}
Note that the promotion step corresponds to shuffling the $\ybox$ past the rest of the rectified tableau. Thus, evacuation-shuffling corresponds to conjugating the promotion operator (on skew tableaux) by rectifying the tableau. Likewise, $\omega$ is the \textit{commutator} of promotion and rectification. \footnote{We note, however, that $\omega$ is not a commutator in the sense of group theory, since it involves maps between two different sets. In particular, as computed in Theorem \ref{thm:intro-parity}, $\omega$ need not be an even permutation.}

\section{A local algorithm for evacuation-shuffling} \label{sec:local-esh}
%%%%%%%%%%%%%%%%%%%%%%%%%
%%%%% PIERI CASE    %%%%%
%%%%%%%%%%%%%%%%%%%%%%%%%

We will now define \newword{local evacuation-shuffling}, %\[\lesh : \LRyb \to \LRby,\]
a local rule for computing $\esh$.
This section is devoted to the definition of the algorithm and proofs of its elementary properties. In Section \ref{sec:main-result}, we will prove that local evacuation-shuffling is the same as $\esh$.

The base case of the algorithm is the \newword{Pieri case}, where $\beta$ is a one-row partition. In this case, $\esh$ was computed in Theorem 5.10 of \cite{bib:Levinson}, and we recall it here. We will give an alternative proof of the Pieri case in Section 4, in part because the complete algorithm relies heavily on our understanding of it.

\begin{theorem}[Pieri case]\label{thm:Pieri}
Let $\beta$ be a one-row partition. Then $\esh(\ybox,T)$ exchanges $\ybox$ with the nearest $1 \in T$ \emph{prior to it} in reading order, if possible. If, instead, the $\ybox$ precedes all $1$'s in reading order, $\esh$ exchanges $\ybox$ with the \emph{last} $1$ in reading order (a \newword{special jump}).
\end{theorem}

We give two examples, illustrating the possible actions of $\esh$ and the more familiar $\sh$. 

\begin{enumerate}
\item If the skew shape contains a (necessarily unique) vertical domino:
\[{\small \young(::\x11,:11,1)} \hspace{0.2cm}\stackrel{\xrightarrow{\esh}}{\xleftarrow[\,\sh\,]{}} \hspace{0.2cm}{\small \young(::111,:1\x,1)}\]
\item Otherwise, the action of $\esh \circ \sh$ cycles the $\ybox$ through the rows of $\gamma^c/\alpha$:
%\item Suppose $\gamma^c/\alpha$ is a horizontal strip having $r$ nonempty rows.  Then each of $\LRyb$ and $\LRby$ have $r$ elements, since the $\ybox$ must be at the left or right end of a row.  The map $\esh$ moves the $\ybox$ down one row, or to the top row if the $\ybox$ is at the bottom.  The map $\sh$ moves the $\ybox$ from the right to the left end of its row.
\[{\small\young(:::\x11,:11,1)} \hspace{0.2cm}\xrightarrow{\esh} \hspace{0.2cm}{\small\young(:::111,:1\x,1)}
\hspace{0.2cm}\xrightarrow{\sh} \hspace{0.2cm} {\small\young(:::111,:\x1,1)}\]
\[{\small\young(:::111,:11,\x)} \hspace{0.2cm}\xrightarrow{\esh} \hspace{0.2cm}{\small\young(:::11\x,:11,1)}
\hspace{0.2cm}\xrightarrow{\sh} \hspace{0.2cm} {\small\young(:::\x11,:11,1)}\]
\end{enumerate}
%\end{theorem}

\subsection{The algorithm}
%%%%%%%%%%%%%%%%%%%%%%%%%
%%%%% THE ALGORITHM %%%%%
%%%%%%%%%%%%%%%%%%%%%%%%%

We now give the definition of the local algorithm.

\begin{definition}\label{def:algorithm}
Let $(\ybox, T) \in \LRby$. We define \textit{local evacuation-shuffling},
\[\lesh : \LRby \to \LRyb,\]
by the following algorithm.
  
  \begin{itemize}
    \item \textbf{Phase 1.} If the $\ybox$ does not precede all of the $i$'s in reading order, switch $\ybox$ with the nearest $i$ \emph{prior} to it in reading order. Then increment $i$ by $1$ and repeat Phase 1. \vspace{0.2cm}

    If, instead, the $\ybox$ precedes all of the $i$'s in reading order, go to Phase 2. \vspace{0.2cm}
    
    \item\textbf{Phase 2.} If the suffix from $\ybox$ is not tied for $(i,i+1)$, switch $\ybox$ with the nearest $i$ \emph{after it} in reading order whose suffix is tied for $(i,i+1)$. Either way, increment $i$ by $1$ and repeat Phase 2 until $i=\ell(\beta)+1$.
  \end{itemize}
\end{definition}
\begin{remark}[Alternate description of Phase 2]\label{rmk:alternate-phase-2}
We will sometimes use the following equivalent description of Phase 2, which we call the \newword{step-by-step} version of Phase 2:

\begin{itemize}
\item {\bf Phase 2$'$ (step-by-step).}
If the suffix from $\ybox$ is not tied for $(i,i+1)$, switch $\ybox$ with the nearest $i$ after it in reading order. Repeat this step until the suffix becomes tied for $(i,i+1)$. Then increment $i$ and repeat Phase 2$'$.
\end{itemize}
\end{remark}

\begin{remark}
Phase 1 is identical to the Pieri case \emph{unless} the Pieri case calls for a special jump.
\end{remark}
Note that in Phase 2, it is not obvious that we can find any $i$ with suffix tied for $(i,i+1)$. We show below, however, that $T$ remains ballot (and semistandard) throughout the algorithm. Consequently, the topmost $i$ is such a square (or $\ybox$ itself, if $\ybox$ is above this $i$).

In Phase 1, $\ybox$ moves down and to the left; in Phase 2 (or 2$'$), $\ybox$ instead moves to the right and up.  We refer to the squares occupied by the box during the step-by-step algorithm as the \newword{evacu-shuffle path}. See Figures \ref{fig:antidiagonal} and \ref{fig:antidiag-evacu-path} for examples.

\begin{remark}[Algorithmic complexity]
Non-local evacuation-shuffling, as defined in Section \ref{sec:case-of-interest}, has running time $O(|\alpha| \cdot b)$, where $b = \ell(\beta) + \ell(\beta^*)$. The local algorithm does not involve the $\alpha$ shape directly and is faster, with running time $O(b)$. Computing the entire orbit decomposition of $\omega$ on $\LRyb$, using the local algorithm, therefore takes $O(b \cdot c_{\alpha,\ybox,\beta,\gamma}^\rect)$ steps. See Corollary \ref{cor:running-time}.
\end{remark}

\begin{definition} 
  We use the following terminology for the $i$-th step of $\lesh$: \\
 
  $\pieri_i$ -- a \newword{regular Pieri jump}, a Phase 1 move in which the $\ybox$ moves down-and-left.

  $\vertical_i$ -- a \newword{vertical slide}, a Phase 1 move in which the $\ybox$ moves strictly down.

  $\jump_i$ -- a \newword{Phase 2 jump}, a move in Phase $2$ involving the $(i,i+1)$ suffixes. \\

\noindent When using Phase 2$'$, we will index moves by their position along the evacu-shuffle path. We write: \\

  $\cpieri_j$ -- a \newword{conjugate Pieri jump}, a Phase 2 move in which the $\ybox$ moves up-and-right.
  
  $\horiz_j$ -- a \newword{horizontal slide}, a Phase 2 move in which the $\ybox$ moves strictly right. \\

\noindent Thus a Phase 2 jump consists of, in general, a possibly empty sequence of conjugate Pieri moves and horizontal slides.

We also say that $s$ is the \newword{transition step} if the algorithm switches to Phase 2 while $i = s$. If the algorithm remains in Phase 1 throughout, we say the transition step is $s=\ell(\beta)+1$.
\end{definition}

%\begin{itemize}
%    \item $\pieri_i$ -- denotes a \newword{regular Pieri jump}, a move in Phase $1$ in which the box moves strictly down and to the left past the horizontal strip of $i$'s, as in the non-special jumps of case 2 of Theorem \ref{thm:Pieri}.
%    \item $\vertical_i$ -- denotes a \newword{vertical slide}, a move in Phase $1$ in which the box moves directly downwards one square past the horizontal strip of $i$'s, as in case 1 of Theorem \ref{thm:Pieri}.
%    \item $\horiz_j$ -- denotes a \newword{horizontal slide}, a Phase 2 move in which the $\ybox$ slides one step to the right, at the $j$th step in the evacu-shuffle path.
%    \item $\cpieri_j$ -- denotes a \newword{conjugate Pieri jump}, a Phase 2 move in which the $\ybox$ moves strictly up and to the right, at the $j$th step in the evacu-shuffle path.
%    \item $\jump_i$ -- denotes a \newword{Phase 2$'$ jump}, a move in Phase $2'$ involving the $(i,i+1)$ suffixes.
%  \end{itemize}
%  We also say that $s$ is the \newword{transition step} if the algorithm switches to Phase 2 while $i = s$. If the algorithm remains in Phase 1 throughout, we say the transition step is $s = \ell(\beta)+1$.
%\end{definition}

\subsection{Examples}
%%%%%%%%%%%%%%%%%%%%%%%%%
%%%%% EXAMPLES      %%%%%
%%%%%%%%%%%%%%%%%%%%%%%%%
We give two examples of our algorithm.  For an online animation, see \cite{bib:Gillespie}.

\begin{example} \label{exa:first-evacu-shuffle}
  Let $$T=\young(::::::111,:::\x 1122,:::1223,:::334,::44,235).$$ We compute $\lesh(\ybox,T)$.  We start in Phase 1 with $i=1$, and do a vertical slide past the $1$'s, then a regular Pieri jump past the $2$'s:
  
  $$\young(::::::111,:::\x 1122,:::1223,:::334,::44,235)\xrightarrow{\vertical_1}
    \young(::::::111,:::11122,:::\x 223,:::334,::44,235)\xrightarrow{\pieri_2}
    \young(::::::111,:::11122,:::2223,:::334,::44,\x35)$$
    \vspace{0.2cm}
    
    Since the $\ybox$ now precedes all the $3$'s in reading order, we transition to Phase 2. We look for the first $3$ after the $\ybox$ (or $\ybox$ itself) whose $(3,4)$-suffix is tied. We interchange the $\ybox$ with that $3$. We repeat for $4$ (interchanging the $\ybox$ with the last $4$, in this case). For $5$, the $(5,6)$-suffix of the $\ybox$ is already tied, since the $\ybox$ is past all the $5$'s. Thus the $\ybox$ does not move further. Phase 2 is as follows:

  $$\young(::::::111,:::11122,:::2223,:::334,::44,\x35)\xrightarrow{\jump_3}
    \young(::::::111,:::11122,:::2223,:::334,::44,3\x 5)\xrightarrow{\jump_4}
    \young(::::::111,:::11122,:::2223,:::33\x,::44,345) = \lesh(T)$$
        \vspace{0.2cm}
        
\noindent Note that $\jump_3$ corresponds in the step-by-step algorithm to $\horiz_3$, and the portion of the evacu-shuffle path corresponding to $\jump_4$ is the sequence of moves $\cpieri_4,\horiz_5, \cpieri_6$.

We will see later (Corollary \ref{cor:transition-step}) that the transition step of $s=3$ indicates that the partition $\beta = (6,5,4,3,1)$ has an outer co-corner in its third row, and that the evacu-shuffle path formed by the step-by-step algorithm therefore has $s + \beta_s = 7$ boxes, including both endpoints.
%
%    Finally, we use jeu de taquin to compute $\sh(\lesh(T))$, to find $$\omega(T) = \young(::::::111,:::\x 1122,:::1223,:::233,::44,345)$$
\end{example}

\begin{example}[Vertical Pieri case]
  As another example, we illustrate the action of $\omega = \sh \circ \lesh$ in the transpose of the Pieri case, where the skew shape is a vertical strip and $\beta=(1,1,\ldots,1)$ is a single column.
  
  Let $$T=\young(::1,:2,:3,\x,4).$$  The tableau is already in Phase $2$ at the step $i=1$.  Since the $(1,2)$-suffix and the $(2,3)$-suffix of the $\ybox$ are already tied, the next step in the evacu-shuffle path is a $\cpieri$ move that interchanges the $\ybox$ with the $3$.  At this point all higher suffixes are tied, and we are done.  For the shuffle step, the box then slides up the second column via jeu de taquin. We find: $$\omega(T)=\young(::1,:\x,:2,3,4).$$  The box continues moving from one column to the next in the until it reaches the top. For the final tableau, the evacuation shuffle consists only of Phase 1 moves and returns to $T$.  The $\omega$-orbit of $T$ is therefore:

$$\young(::1,:2,:3,\x,4)\xrightarrow{\ \omega\ } \young(::1,:\x,:2,3,4) \xrightarrow{\ \omega\ } \young(::\x,:1,:2,3,4)\xrightarrow{\ \omega\ } \young(::1,:2,:3,\x,4).$$

\end{example}

\subsection{Properties preserved by local evacuation shuffling}
%%%%%%%%%%%%%%%%%%%%%%%%%
%%%%% PROPERTIES    %%%%%
%%%%%%%%%%%%%%%%%%%%%%%%%

We will require the fact that the tableau remains semistandard and ballot during local evacuation-shuffling, and moves past the strip of $i$'s at the $i$th step of the default algorithm.

\begin{theorem}\label{thm:ballotness}
Let $T$, including the $\ybox$, be a tableau that appears in the step-by-step (Phase 2$'$) computation of $\lesh(\ybox,T_1)$ for some pair $(\ybox,T_1)\in \LRyb$.  Then:
\begin{itemize}
\item[(1)] Omitting the $\ybox$, the reading word of $T$ is ballot.
\item[(2)] Omitting the $\ybox$, the rows (resp. columns) of $T$ are weakly (resp. strictly) increasing.
\item[(3)] If $T=T_i$ appears just before the $i$-th step of the \emph{default} (not step-by-step) algorithm, then the $\ybox$ is an outer co-corner of the collection of squares in $T$ having entries $1,\ldots,i-1$, and an inner co-corner of the squares in $T$ having entries $i,\ldots,t$.
\end{itemize}
\end{theorem}

%To prove this, we first require a technical lemma.  

%\begin{lemma}
%For each $i$, let $T_i$ be the tableau just before the $i$th step of the local evacuation shuffle (using the default Phase 2).  Write $\young(:N,W\x E,:S)$ for the squares around $\ybox$ in $T_i$. If nonempty, these satisfy the following constraints:
%\begin{itemize}
%\item If the $i$-th step is in phase 1: $N, W \leq i-1$ and $E,S \geq i$.
%\item If the $i$-th step is in phase 2: $N, W \leq i-1$; $E \geq i$; and $S \geq i+1$.
%\end{itemize}
%The tableau at the transition step counts as both phases (i.e. as phase 2).
%\end{lemma}

\begin{proof}
We first show that the conditions hold for the tableaux occurring via the default algorithm.  Let $T_i$ be the tableau before the $i$-th move, using the default description of Phase 2. Conditions (1)-(3) are clearly satisfied in the starting tableau $T_1$.  Now let $i\ge 1$ and suppose $T=T_{i+1}$.  Assume for induction that the conditions are satisfied for $T_{i}$.

\textbf{Case 1:} Suppose $T_i$ is in Phase 1, i.e., a Phase 1 move is applied to $T_i$ to get $T_{i+1}$.

We first check that $T_{i+1}$ satisfies (2) and (3).  Since the move from $T_i$ to $T_{i+1}$ was a vertical slide or Pieri move that switches the $\ybox$ with the next $i$ in reverse reading order, the old position of the $\ybox$ is now filled with an $i$.  This position must satisfy (2) in $T_{i+1}$, since $T_i$ satisfied condition (3) and the only way an $i$ could be below this square in $T_i$ is if a vertical slide occurs (in which case it's no longer there in $T_{i+1}$).  All other rows and columns clearly still satisfy (2), and by the definition of the Phase 1 moves we see that $T_{i+1}$ satisfies (3) as well.

%If $S = i$, the move is vertical; if $S \geq i+1$, the move is a Pieri jump. In both cases, the $\ybox$ switches with an $i$ in $T_{i+1}$. Condition (2) is preserved at the old location (by induction) and (3) applies at the new one (since $T_i$ is semistandard). If $T_{i+1}$ is the transition tableau, we also see immediately that the new $S$ has $S \geq i+2$, as required.

We now check that $T_{i+1}$ satisfies (1). The effect of the move on the reading word is to move a single $i$ entry later in the word, so we need only check that the $(i-1,i)$-subword is still ballot after the move. This is vacuous if $i=1$, so assume $i \geq 2$.

Let $x$ and $z$ be the positions of $\ybox$ in $T_i$ and $T_{i+1}$ respectively.  The only suffixes affected by the Phase 1 move are the suffixes of squares $y$ that occur weakly after $x$ and strictly before $z$ in reading order.  Let $y$ be such a square.  Since $i\ge 2$, we know $x$ contained an $i-1$ in $T_{i-1}$, and that this $i-1$ moved later in the reading word to form $T_i$.  Since the suffix of $y$ was ballot in $T_{i-1}$, it follows that in $T_i$ the suffix of $y$ has at least one more $i-1$ than $i$.  Thus the suffix of $y$ formed by replacing $x$ by $i$ is ballot as well.

\textbf{Case 2:}  Suppose $T_i$ is in Phase 2, i.e., a Phase 2 move will be applied to $T_i$ to get $T_{i+1}$.

We first show (2).  If the $\ybox$ moves, the condition (3) on $T_i$ shows that the old location, say $x$, of $\ybox$ becomes semistandard when filled with $i$ in $T_{i+1}$, except possibly if the square just below $x$ is also filled with $i$.  If the previous move was Phase 1 or if $i=1$, then this is impossible since then we would stay in Phase 1 using a vertical slide.  

Otherwise, if the previous move was Phase 2, assume for contradiction that the square below $x$ contains $i$.  Then it contained $i$ in $T_{i-1}$ and $T_i$ as well.  Consider the leftmost $i-1$ in $x$'s row in $T_i$, or $\ybox$ if there are no other $i$'s. Let $y$ be the square below it, demonstrated with $i=2$ below:
\[\young(11\cdots 1\x,y2\cdots 22)\]
We have $y=i$ since the tableau is semistandard. By definition, the suffix from $\ybox$ in $T_{i}$ is tied for $(i-1,i)$. Hence, the \emph{weak} suffix starting at $y$ is not ballot for $(i-1,i)$. This contradicts ballotness of $T_{i-1}$.  Thus $T_{i+1}$ satisfies (2).

To check (3), we wish to show that the new position of $\ybox$ in $T_{i+1}$, when filled with $i$, was an outer corner of the strip of $i$'s in $T_i$.  Indeed, if not then since the $i$'s form a horizontal strip it must be directly to the left of another $i$, which contradicts ballotness of $T_i$ (since the weak suffix of the $\ybox$ is already tied for $(i,i+1)$, and so the suffix of the $i$ to the right would not be ballot).  Since $T_i$ is semistandard, the $\ybox$ is then also an inner co-corner of the entries larger than $i$ in $T_{i+1}$.

Finally, we check (1), that $T_{i+1}$ is ballot. If the reading word is unchanged by the move, we are done. Otherwise, it has moved a single $i$ earlier in the word. In the latter case we only need to check that the $(i,i+1)$-subword is still ballot after the move.  
  
  By definition, we switch the $\ybox$ (say in position $x$) with the first $i$ after it whose $(i,i+1)$-suffix is tied (say in position $z$).  This does not affect any suffix starting before $x$ or weakly after $z$, so let $y$ be a square between $x$ and $z$ in reading order, possibly equal to $x$.  If $y$ contains an $i$ in $T_i$, its suffix is not tied before the move, hence has strictly more $i$'s than $i+1$'s.  Thus the suffix remains ballot after losing an $i$. Otherwise, let $y'$ be the closest square containing $i$ prior to $y$ in the reading word.  Since $T_i$ is semistandard, the suffix from $y$ contains as many $i$'s, and at most as many $i+1$'s, as the suffix from $y'$.  Since the latter remains ballot, the former does as well.

This completes Case 2.

Finally, to deduce properties (1) and (2) for the step-by-step algorithm, consider that $\jump_i$ corresponds to moving the $\ybox$ past a portion of the horizontal strip of $i$'s. Since the tableaux before and after the jump are semistandard and ballot, it's easy to check that each intermediate tableau (arising in Phase 2$'$) is semistandard and ballot as well. \qed
\end{proof}

\subsection{Reversing the algorithm}
%%%%%%%%%%%%%%%%%%%%%%%%%
%%% REVERSE ALGORITHM %%%
%%%%%%%%%%%%%%%%%%%%%%%%%
We now give an algorithm that undoes $\lesh$.

\begin{definition}\label{def:reverse-algorithm}
  We define the \textit{reverse (local) evacuation-shuffle} of $(T',\ybox) \in \LRby$ to be the pair $(\ybox,T)$ of the same total shape, defined by the following algorithm.
  \begin{itemize}
  \item Set $i=t$.
  \item \textbf{Reverse Phase 2.}   If the suffix of the $\ybox$ has strictly more $i$'s than $i+1$'s, go to Reverse Phase 1.
Otherwise, choose the first $i$ (or $\ybox$) \emph{prior} to the $\ybox$ in reading order whose weak suffix (including itself) has exactly as many $i-1$'s as $i$'s.  If no such entry exists, choose the very first $i$ in reading order.  Interchange this choice of $i$ (or $\ybox$) with the $\ybox$. Decrement $i$ and repeat this step.
  
  \item \textbf{Reverse Phase 1.} Switch $\ybox$ with the nearest $i$ \emph{after} it in reading order.  Decrement $i$ and repeat this step until $i=0$.
  \end{itemize}
\end{definition}

\begin{theorem}\label{thm:reverse-algorithm}
  Reverse local evacuation shuffling is the inverse of local evacuation shuffling.
\end{theorem}

\begin{proof}
Let $(\ybox,T) \in \LRby$ and put $\lesh(\ybox,T) = (T',\ybox)$. We show that the reverse evacuation shuffle of $(T',\ybox)$ is equal to $(\ybox, T)$. Since $\lesh$ is a function between sets of the same cardinality, we will be done.

  Let $\beta=(\beta_1,\ldots,\beta_t)$ be the content of $T$.  Suppose the local evacuation shuffle of $(\ybox,T)$ consists of $k$ moves in Phase 1 and $t-k$ in Phase 2.  If $k=t$ then the last step is still in Phase 1, coming from a Pieri move across the horizontal strip of $t$'s.  Then after this move, the $(t,t+1)$ suffix is not tied because there are no $t+1$'s and there is at least one $t$ after the $\ybox$.  Thus there is no Reverse Phase 2 when applying the reverse algorithm; it starts immediately in Reverse Phase 1.
  
  Otherwise, if $k<t$, the local evacuation shuffle ended with a sequence of $\jump$ moves.  We show inductively that each Reverse Phase 2 step undoes a Phase 2 step in succession. Suppose that reverse-shuffling past the $t,t-1,\ldots,t-i+1$ leaves us at the end of step $t-i$ of $\lesh$, and that step $t-i$ was a $\jump$ move.
  
  In what remains, let $r=t-i$. Let $T_{r}$ and $T_{r+1}$ be the respective tableaux before and after the $\jump_{r}$ step, and let $s$ and $s'$ denote the squares that contain the $\ybox$ in $T_{r}$ and $T_{r+1}$ respectively.  Then $T_{r+1}$ is formed by switching the $\ybox$ (from position $s$) with the first $r$ after it (in position $s'$) whose $(r,r+1)$ suffix is tied.  The Reverse Phase 2 step, backwards past the $r$ strip, would take the $\ybox$ and switch it with either the first $r$ to the left whose weak $(r-1,r)$ suffix is tied, or the very first $r$ in reading order.  We wish to show that this $r$ is in location $s$ in $T_{r+1}$.
  
  First suppose that the $(r-1)$st step was also a $\jump$ move.  Then in $T_{r}$, the $(r-1,r)$-suffix of the $\ybox$ is tied.  So, in $T_{r+1}$, the \textit{weak} suffix starting at square $s$ is tied for $(r-1,r)$ as well.  Assume for contradiction that there were another $r$ between $s$ and $s'$ in reading order whose weak $(r-1,r)$ suffix is tied in $T_{r+1}$.  Then in $T_{r}$, that suffix would have strictly more $r$'s than $r-1$, contradicting ballotness of $T_{r}$ (see Lemma \ref{thm:ballotness}).  Thus the $r$ in square $s$ is the first $r$ to the left of the $\ybox$ in reading order in $T_{r+1}$ whose weak $(r-1,r)$ suffix is tied, and so the reverse process moves the $\ybox$ back to square $s$.
  
  Otherwise, if the $(r-1)$st step was a Pieri (Phase 1) move, then in $T_{r}$, the $(r-1,r)$-suffix of the $\ybox$ cannot be tied, since $T_{r-1}$ is ballot and we replaced the $\ybox$ with another $r-1$, which adds to that suffix.  Notice also that since the $(r)$th step is the first step in Phase 2, the $\ybox$ must precede all $r$'s in reading order in $T_{r}$.  Thus square $s$ is the leftmost $r$ in reading order in $T_{r+1}$, and no other $r$ can have weakly tied $(r-1,r)$ suffix by the same ballotness argument as above.  It follows that the reverse move does switch the $\ybox$ with the $r$ in square $s$ in this case as well.
  
  This shows that the $\jump$ moves are undone by the Reverse Phase 2 moves, and that the reverse algorithm switches to Reverse Phase 1 exactly after undoing all the forward Phase 2 moves.  It is easy to see that a Reverse Phase 1 move is the inverse of a forward Phase 1 move as well, so this algorithm reverses the local evacuation shuffling algorithm. \qed
\end{proof}

\begin{remark}
  The algorithm in Definition \ref{def:reverse-algorithm} reverses the ordinary (not step-by-step) algorithm.  To reverse the step-by-step algorithm, we simply break each Reverse Phase 2 jump into smaller steps, interchanging $\ybox$ with each $i$ that precedes it in succession until it reaches the first $i$ whose suffix had exactly as many $i$'s as $i-1$'s (before switching it with $\ybox$).
\end{remark}

\section{Proof of local algorithm}\label{sec:main-result}

%%%%%%%%%%%%%%%%%%%%%%%%%%%%%%%
%%%% Proof of Main Result %%%%%
%%%%%%%%%%%%%%%%%%%%%%%%%%%%%%%

\begin{figure}[t]
\begin{center}
  \includegraphics[scale=1]{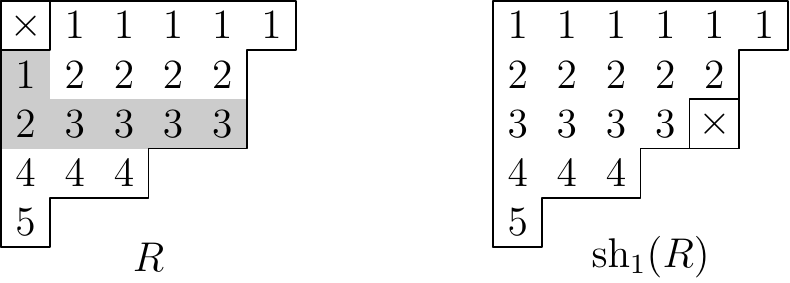}
\end{center}
\caption{An example of a rectified tableau $R$ with transition step $s=3$.  The rectification path of the box is down to row $s$ and then directly right.}
\label{fig:R-diagram}
\end{figure}

In this section we prove the following:

\begin{theorem}\label{thm:main-theorem}
  Local evacuation-shuffling is the same map as evacuation-shuffling, that is, for any $(\ybox,T)\in \LRby$,
  \[\lesh(\ybox,T) = \esh(\ybox,T).\]
\end{theorem}

%This can be proven by the inductive approach used in the Pieri case, but we instead provide a cleaner proof relying on the theory of dual equivalence classes.

The main idea is as follows.  In computing $\esh$, when we first rectify $(\ybox,T)$, we obtain a tableau $R$ of the form shown in Figure \ref{fig:R-diagram}.  In particular, the $\ybox$ is in the inner corner and the total shape of $\ybox \sqcup R$ is formed by adding an outer co-corner to $\beta$ in some row $s$.  %In any row $i\le s$ the first square is $i-1$ (or $\ybox$ if $i=1$) and the remaining squares contain $i$; in any row $i>s$ the entries are all $i$.

When shuffling the $\ybox$ past $R$, the $\ybox$ follows a path directly down to row $s$ and then directly over to the end of row $s$, as shown.  It turns out that this corresponds to a more refined process in which we shuffle the $\ybox$ past rows $1,2,\ldots,s-1$, then shuffle it past the $\beta_s$ vertical strips formed by greedily taking vertical strips from the right of the bottom $l(\beta)-s+1$ rows of the tableau.  We call this decomposition into horizontal and vertical strips the \newword{$s$-decomposition}, as illustrated in Example \ref{ex:s-decompositions}.

%We will show that each Phase 1 move of the algorithm corresponds to a single move of the $\ybox$ past a horizontal strip, and the Phase 2 moves similarly can be described as shuffling $\ybox$ past each of the vertical strips in the $s$-decomposition.

We show that each step of Phase 1 of $\lesh$ corresponds to a single move of the $\ybox$ past a horizontal strip, and that the transition step is $s$. We then show, using the \emph{antidiagonal symmetry} suggested by Figure \ref{fig:antidiagonal}, that the movements of the $\ybox$ during Phase 2 correspond similarly to shuffles past each of the $s$-decomposition's vertical strips.

\begin{definition} \label{def:conjugate-pieri}
  Let $V$ be a vertical strip, i.e., no row of $V$ contains more than one entry.  Let $\ybox$ be an inner co-corner of $V$.  Then we define the \newword{conjugate move} to be the action of switching the location of the $\ybox$ with the square of $V$ that comes directly \emph{after} it in reading order.
\end{definition}

%\begin{figure}[h]
%\begin{center}
% \includegraphics{s-decomposition.pdf}
% \end{center}
% \caption{\label{fig:s-decomposition} The $3$-decomposition of $\beta$.  It consists of the first $2$ horizontal strips, then the $4$ vertical strips formed by greedily taking vertical strips from the right of the tableau formed by the rows weakly below row $3$.}
%\end{figure}

\subsection{\texorpdfstring{$s$}{s}-decompositions}

We formalize the notion of an $s$-decomposition and extend it to an arbitrary Littlewood-Richardson tableau as follows.

\begin{definition}[$s$-decompositions]
Let $1 \leq s \leq \ell(\beta)+1$. 

\begin{enumerate}
\item Let $\beta'$ be obtained by deleting the first $s-1$ rows of $\beta$. Let $r_1, \ldots, r_{s-1}$ be one-row partitions with lengths the first $s-1$ rows of $\beta$, and let $c_s, \ldots, c_t$ be one-column partitions of lengths given by the columns of $\beta'$ in reverse order. (Here $t = \beta_s + s - 1$.)  We say that $(r_1,\ldots,r_{s-1},c_s,\ldots,c_t)$ is the \newword{$s$-decomposition} of the shape $\beta$.

\item Let $T \in \LR_\mu^\lambda(\beta)$ be a ballot SSYT. The $s$\newword{-decomposition of} $T$ is the decomposition of $T$ into its first $s-1$ horizontal strips $H_1,\ldots,H_{s-1}$ where $H_i$ consists of the entries labeled $i$ in $T$, followed by $\beta_s$ vertical strips $V_s, \ldots, V_t$, where $V_{t+1-i}$ contains the $i$-th-from-last instance (when possible), in reading order, of each of the entries $j \geq s$.

\end{enumerate}
\end{definition}

The $s$-decomposition of the highest weight filling of $\beta$ will be of particular importance.

\begin{example}\label{ex:s-decompositions}
The $3$-decomposition of the tableau $T$ used in Example \ref{exa:first-evacu-shuffle} is shown in Figure \ref{fig:s-decomposition-color} (note that $3$ is the transition point for the initial position of the $\ybox$ in that example). Notice that this corresponds to the $s$-decomposition of the rectified tableau of shape $\beta$ shown in Figure \ref{fig:rectified-s-decomposition}.
\end{example}
%\begin{center}
%\includegraphics[width=15.8cm]{skew-s-decompositions-6.pdf}
%\end{center}

\begin{figure}[h] 
\begin{center}
\includegraphics[width=15.8cm]{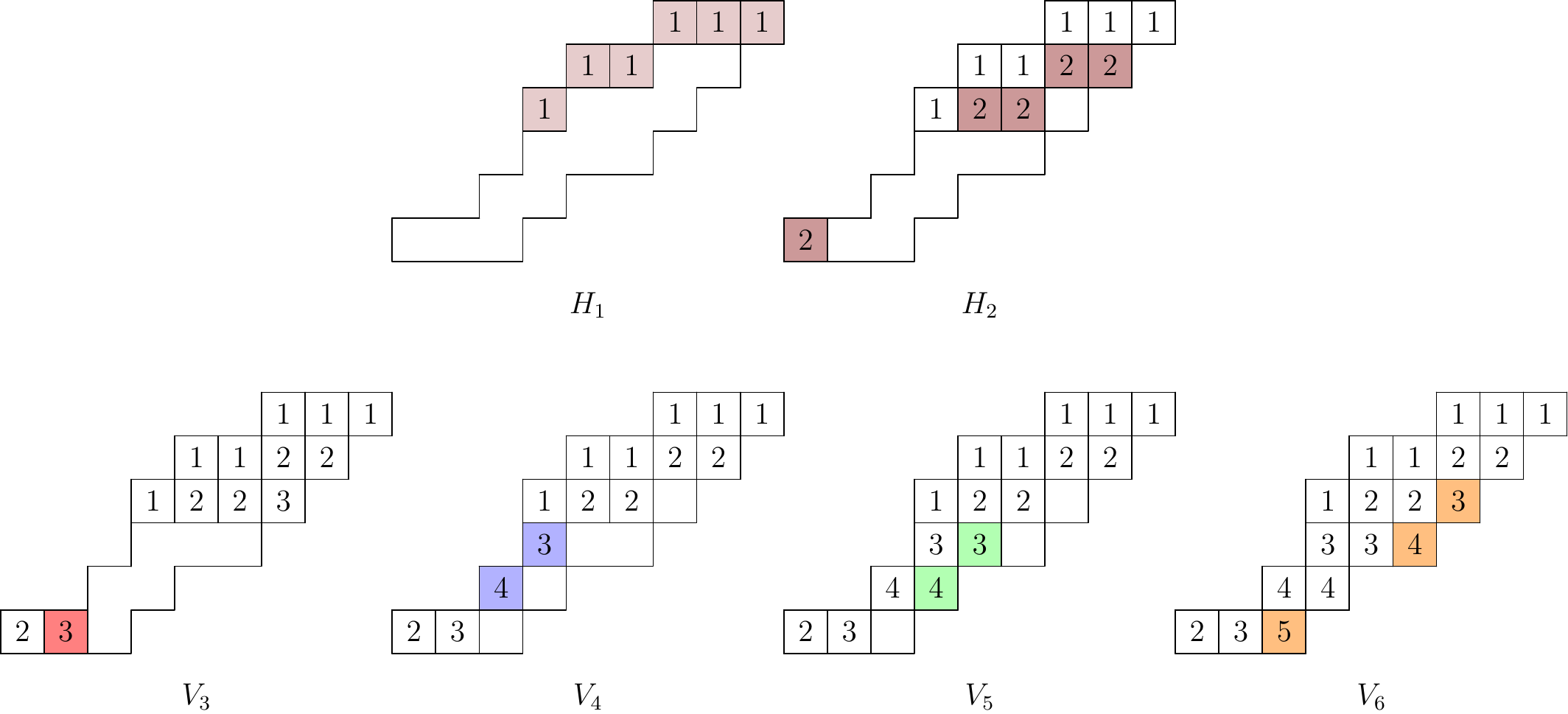}
\caption{\label{fig:s-decomposition-color} The $3$-decomposition into horizontal and vertical strips of the tableau discussed in Example \ref{ex:s-decompositions}.}
\end{center}
\end{figure}

\begin{figure}[h]
\begin{center}
\hspace{-1cm}\includegraphics[height=2.6cm]{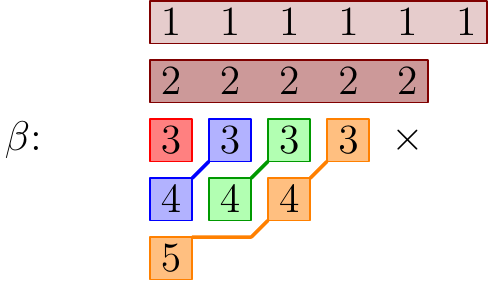} \hspace{1cm}
\includegraphics[height=2.6cm]{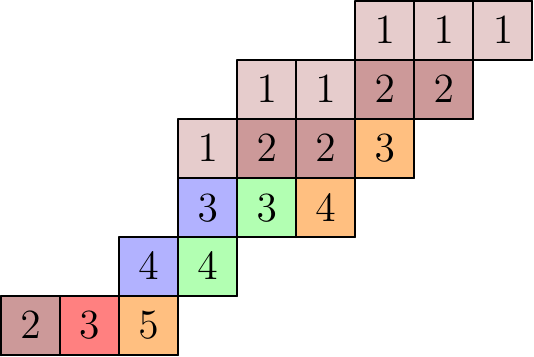}
\caption{\label{fig:rectified-s-decomposition} At left, the $3$-decomposition of $\beta$, where $\beta$ is the rectification shape of  the tableau $T$ from Example \ref{ex:s-decompositions}.  The $3$-decomposition of $T$ is color-coded at right.}
\end{center}
\end{figure}

\begin{lemma}\label{lem:5-facts}
  Let $(\ybox,T)\in \LRyb$, and let $H_1,\ldots,H_{s-1},V_s,\ldots,V_{t}$ be its $s$-decomposition.  Then
  \begin{enumerate}
    \item[(i)] $H_1,\ldots,H_{s-1}$ are horizontal strips with $H_i$ extending $H_{i-1}$ for all $i$.
    \item[(ii)] $V_s,\ldots,V_{t}$ are vertical strips, with $V_s$ extending $H_{s-1}$ and $V_j$ extending $V_{j-1}$ for all $j$.
    \item[(iii)] For any $i$, $H_i$ rectifies to the $i$th row in $\rectify(T)$.
    \item[(iv)] For any $i$, $V_{t-i+1}$ rectifies to the $i$th vertical strip in the $s$-decomposition of $\rectify(T)$.
  \end{enumerate}
\end{lemma}

\begin{proof}
  To prove (i) and (iii), note that $H_i$ rectifies to the $i$th row  of the highest weight filling of $\beta$ since it is filled with all $i$'s in $T$.  They form a horizontal strip because $T$ is semistandard.
  
  To prove (iv), let $j\ge s$.  If we order the $j$'s in the highest weight filling of $\beta$ in reading order, then they must occur in that order in $T$ as well, since the reading word of $T$ is Knuth equivalent to that of its rectification and Knuth moves cannot switch equal-valued entries.  (See \cite{bib:Fulton} for an introduction to Knuth equivalence.)  The vertical strip $\beta^{(t-i+1)}$ in the rectified picture consists of the $i$th copies from the right of each such entry $j$, and so in $T$ the entry $j$ occurring in $V_{t-i+1}$ is still the $i$th from the end.
  
  For (ii), since the reading word of $T$ is ballot, the $i$th-to-last copy of $j$ must occur strictly after the $i$th-to-last copy of $j+1$ for any $j$, and since the tableau is semistandard, this $j+1$ cannot appear strictly to the left of the $j$.  It follows that the $j+1$ in $V_{t-i+1}$ appears in a row strictly below the $j$ in $V_{t-i+1}$ for each $j$.  Therefore, $V_{t-i+1}$ is a vertical strip for all $i$.  Since each of the entries $j\ge s$ appears in $V_k$ before $V_{k+1}$ for all $k$, the strips must extend each other as well.  Finally, $V_s$ extends $H_{s-1}$ because it consists of entries larger than $s-1$ and is the first of each of those entries in its row. \qed
\end{proof}

\begin{remark}
Lemma \ref{lem:highwt-lowwt} follows from Lemma \ref{lem:5-facts} in the case $s=1$. To see this, observe that the $s$-decomposition is, in particular, preserved by jeu de taquin slides applied to $T$. If we \emph{anti-rectify} $T$ to a tableau of shape $\rect/\beta^c$, the explicit description of the entries of $V_{t-i+1}$ shows that it forms precisely the $i$-th-rightmost column of $\rect/\beta^c$.
\end{remark}

Lemma \ref{lem:5-facts} allows us to factor Littlewood-Richardson chains into longer chains based on the $s$-decomposition. In particular, for a horizontal strip $H_i$ or vertical strip $V_j$ in an $s$-decomposition, let $\overline{H_i}$ and $\overline{V_j}$ be the corresponding Littlewood-Richardson tableaux of content $r_i$ and $c_j$ respectively formed by decreasing the entries appropriately.  We have the following map.

\begin{definition}
We write $$\iota_s:\LRyb\to \LR(\alpha,\ybox,r_1,\ldots,r_{s-1},c_s,\ldots,c_t,\gamma)$$ by $$\iota_s(T_\alpha,\ybox,T,T_\beta)=(T_\alpha,\ybox, \overline{H_1},\ldots,\overline{H_{s-1}},\overline{V_s},\ldots,\overline{V_{t}},T_\gamma)$$ where $(H_i,V_j)$ is the $s$-decomposition of $T$.  We define $$\iota_s:\LRby\to \LR(\alpha,r_1,\ldots,r_{s-1},c_s,\ldots,c_t,\ybox,\gamma)$$ in a similar fashion.
\end{definition}

Note that $\iota_s$ is injective, because the process of reducing the strips into Littlewood-Richardson tableau can be reversed by increasing the entries of each $\overline{H_i}$ by $i-1$ and increasing those of $\overline{V_j}$ by $s-1$.  We also claim that shuffling any tableau with $T$ is the same as shuffling past each of the $H_i$ and $V_j$ in sequence. This is proven in the two lemmas that follow.

In these lemmas it is helpful to use the language of dual equivalence classes in place of Littlewood-Richardson tableaux (note that $s$-decompositions and the map $\iota_S$ can be similarly defined on dual equivalence classes, by taking the associated classes of the tableaux at each step.)

\begin{lemma}[Extracting horizontal strips] \label{lem:factor-horiz}
Let $\lambda/\mu$ be a skew shape and $\beta = (\beta_1, \ldots, \beta_r)$ a partition. Let $\beta' = (\beta_2, \ldots, \beta_r)$. Consider the concatenation map on dual equivalence classes,
\[\DE_\mu^\lambda(\beta_1, \beta') \to \bigsqcup_{\tau} \DE_\mu^\lambda(\tau), \qquad (D_1, D') \mapsto D_1 \sqcup D',\]
where the union is over $\tau \subseteq \beta'$ with $\tau/\beta'$ a horizontal strip of length $\beta_1$.

There is a unique `factorization' injection, a right inverse to concatenation,
\[\iota_H : \DE_\mu^\lambda(\beta) \hookrightarrow \DE_\mu^\lambda(\beta_1, \beta').\]
It is `compatible with shuffling' in the sense that the following diagram commutes, for any partition $\pi$:
\[\xymatrix{
\DE_\mu^\lambda(\beta, \pi) \ar[r]^-{\iota_H} \ar[d]_{\sh_1} & \DE_\mu^\lambda(\beta_1, \beta', \pi) \ar[d]^{\sh_1 \sh_2} \\
\DE_\mu^\lambda(\pi, \beta) \ar[r]^-{\iota_H} & \DE_\mu^\lambda(\pi, \beta_1, \beta').
}\]
\end{lemma}
We think of $\iota_H$ as `extracting the highest-weight horizontal strip' from the inner edge of the shape.
\begin{proof}
Let $D \in \DE_\mu^\lambda(\beta)$. By the Pieri rule, at least one pair $(D_1, D') \in \DE_\mu^\lambda(\beta_1, \beta')$ has $D_1 \sqcup D' = D$. We wish to define $i_H(D) := D'$.

Suppose $(E_1, E')$ is another such pair. Let $D_\mu$ be the unique dual equivalence class of straight shape $\mu$. Perform shuffles to obtain
\begin{align*}
\sh_2\sh_1(D_\mu, D_1, D') &= (D_{\beta_1}, \tilde{D'}, \tilde{D_\mu}), \\
\sh_2\sh_1(D_\mu, E_1, E') &= (D_{\beta_1}, \tilde{E'}, \tilde{E_\mu}).
\end{align*}
Concatenation is compatible with shuffling, so $\tilde{D_\mu} = \tilde{E_\mu}$, as both correspond to shuffling $D_\mu$ with $D$. Moreover, we have $\tilde{E'}, \tilde{D'} \in \DE_{\beta_1}^\beta(\beta')$, which is a singleton set. (Note that $\beta / \beta_1$ is effectively a straight shape.) So $\tilde{E'} = \tilde{D'}$ and so, after shuffling once more with $\tilde{D_\mu}$, we conclude $(E_1,E') = (D_1,D')$. Finally, $\iota_H$ is compatible with shuffling because concatenation is (and $\iota_H$ is a right inverse to concatenation). \qed
\end{proof}
\begin{lemma}[Vertical strips and outer strips] \label{lem:factor-other}
Let $c$ be the first column of $\beta$, and let $\beta''$ be $\beta$ with $c$ deleted. There are injections
\begin{align*}
\iota_H^\ast : \DE_\mu^\lambda(\beta) &\hookrightarrow \DE_\mu^\lambda(\beta', \beta_1), \\
\iota_V : \DE_\mu^\lambda(\beta) &\hookrightarrow \DE_\mu^\lambda(c, \beta''), \\
\iota_V^\ast : \DE_\mu^\lambda(\beta) &\hookrightarrow \DE_\mu^\lambda(\beta'', c),
\end{align*}
where $\iota_H^\ast$ corresponds to extracting the maximal horizontal strip along the outer (southeast) edge of the shape, and $\iota_V, \iota_V^\ast$ extract maximal \emph{vertical} strips from the inner and outer edges, respectively. Each of these is a right inverse to concatenation and is compatible with shuffling.
\end{lemma}
\begin{proof}
We obtain $\iota_H^\ast$ by rotating tableaux, that is, $\iota_H^\ast(D) = \iota_H(D^R)^R.$ Similarly, we obtain $\iota_V$ by transposing, and $\iota_V^\ast$ by rotating and transposing. \qed
\end{proof}
Notice that rotating and transposing $D$ exchanges $\iota_H$ with $\iota_V^*$. By Lemma \ref{lem:highwt-lowwt}, it follows that the maximal outer vertical strips extracted by $\iota_V^\ast$ are the same as those of the $1$-decomposition of $\beta$. More generally, $\iota_s$ is the composition of several applications of $\iota_H$ and $\iota_V^\ast$: if $D = \DE(T)$, we have
\[\iota_s(T) = \LR \circ (\iota_V^\ast)^{\beta_s} (\iota_H)^{s-1}(D).\]

We now refine evacuation-shuffling by factoring $\esh$ into a sequence of operations $e_1, \ldots, e_{s-1+\beta_s}$, corresponding to an $s$-decomposition.

\begin{definition}
For a fixed $s$, and for $1 \leq i \leq t = s-1+\beta_s$, we define the \newword{partial evacuation shuffle}
  \[e_i : \LR(\alpha, r_1, \ldots, \ybox, r_i, \ldots, c_t, \gamma) \to \LR(\alpha, r_1, \ldots, r_i, \ybox, \ldots, c_t, \gamma)\]
by the composition
\[e_i=(\sh_1\sh_2\cdots \sh_{i+1})\sh_i(\sh_{i+1}\cdots \sh_2\sh_1).\]
(If $i \geq s$, the $r_i$ in the definition above should be replaced by $c_i$.)
\end{definition}
Combinatorially, $e_i$ is a modified version of evacuation shuffling, where:
\begin{enumerate}
\item We rectify the first $i-1$ strips, obtaining a straight shape tableau $B$;
\item We then perform a ``relative'' evacuation-shuffle on $\ybox$ and the $i$-th strip: we rectify them only up to the outer boundary of $B$, then shuffle and un-rectify.
\end{enumerate}

\begin{lemma}\label{lem:partial-esh}
  For any $T\in \LRyb$, and any $s$, we have $$\iota_{s}(\esh(T))=e_t\cdots e_1(\iota_{s}(T)).$$
\end{lemma}

\begin{proof}
Recall that $\esh:\LRyb\to \LRby$ is the composition 
$$\xymatrix{\LRyb \ar[r]^{\sh_2\sh_1} & \LR(\ebox,\beta,\alpha,\gamma) \ar[r]^{\sh_1} &
\LR(\beta,\ebox,\alpha,\gamma) \ar[r]^{\sh_1\sh_2}
& \LRby}$$  

The maps $\iota_H$ and $\iota_V^\ast$ respect shuffling (in the sense stated in Lemmas \ref{lem:factor-horiz} and \ref{lem:factor-other}, translated from dual equivalence classes to the corresponding Littlewood-Richardson tableaux). We thus write
\[\xymatrix{
\LRyb 
  \ar[r]^-{\iota_{s}} \ar[d]^{\sh_2\sh_1}  &
\LR(\alpha,\ebox,r_1, \ldots, c_t,\gamma)     
  \ar[d]^{\sh_{t+1}\cdots\sh_2\sh_1} \\
\LR(\ebox,\beta,\alpha,\gamma) 
  \ar[r]^-{\iota_{s}} \ar[d]^{\sh_1}       & 
\LR(\ebox,r_1, \ldots, c_t,\alpha,\gamma)
  \ar[d]^{\sh_t\cdots\sh_1}          \\
\LR(\beta,\ebox,\alpha,\gamma) 
  \ar[r]^-{\iota_{s}} \ar[d]^{\sh_1\sh_2}  & 
\LR(r_1, \ldots, c_t,\ebox,\alpha,\gamma)
  \ar[d]^{\sh_1\sh_2\cdots\sh_{t+1}}\\ 
\LRby 
  \ar[r]^-{\iota_{s}} & 
\LR(\alpha,r_1, \ldots, c_t,\ebox,\gamma)
}\]

Thus we have 
\[\iota_{s} \circ \esh=(\sh_1\sh_2\cdots \sh_{t+1})(\sh_{t}\cdots\sh_1)(\sh_{t+1}\cdots \sh_2\sh_1)\circ \iota_{s}.\]

We now write out the composition of the $e_i$'s as the reverse-ordered product
\[e_t\cdots e_1=\prod_{i=t}^1(\sh_1\cdots \sh_{i+1})\sh_i(\sh_{i+1}\cdots\sh_1).\]

Notice that, since the shuffles are all involutions, the right-hand term of the $i$-th factor mostly cancels with the left-hand term of the $(i-1)$-st factor. After all such cancellations, we are left with the product
\[(\sh_1\cdots\sh_{t+1})(\sh_{t}\sh_{t+1})(\sh_{t-1}\sh_{t})\cdots (\sh_3\sh_4)(\sh_{2}\sh_{3})(\sh_{1}\sh_2)\sh_1.\]

Recall that $\sh_i$ commutes with $\sh_j$ whenever $|i-j|\ge 2$.  Thus we can move the rightmost $\sh_3$ past the $\sh_1$ next to it, then move the rightmost $\sh_4$ past the $\sh_2$ and $\sh_1$ to its right, and so on. We obtain the product $$(\sh_1\cdots \sh_{t+1})\sh_t\cdots \sh_1 (\sh_{t+1}\cdots \sh_1).$$ This matches our expression for $\esh$ above. \qed
\end{proof}
We emphasize that, for each choice of $s$, we have a \emph{distinct} factorization of $\esh$ into partial evacuation-shuffles as above. In our proof of Theorem \ref{thm:main-theorem}, we cannot use the same choice of $s$ for all $(\ybox,T) \in \LRyb$. Our proof relies on a careful choice of $s$ depending on $(\ybox,T)$, which will make the partial steps $e_i$ correspond to the steps of \emph{local} evacuation-shuffling for the particular pair $(\ybox,T)$. \\

\subsection{The Pieri Case\texorpdfstring{, $\beta = (m)$}{}.}

We now give the proof of Theorem \ref{thm:Pieri}, the Pieri case. We give a more detailed statement:

\begin{theorem}[Pieri case]
Let $\beta=(m)$ be a one-row partition.
\begin{enumerate}
 \item Suppose $\gamma^c/\alpha$ is \emph{not} a horizontal strip. Then $\gamma^c/\alpha$ contains a unique vertical domino; otherwise there is no semistandard filling of $\gamma^c/\alpha$ using a $\ybox$ and $1$'s.

In this case, $\LRyb$ and $\LRby$ have one element each, since the $\ybox$ must be at the top or bottom of the domino. Then $\esh$ slides the $\boxtimes$ down. 

\item Suppose $\gamma^c/\alpha$ is a horizontal strip having $r$ nonempty rows. There is a natural ordering\footnote{Our ordering is the reverse of the ordering used in \cite{bib:Levinson}.} of the tableaux
\[\LRyb = \{L_1, \ldots, L_r\},\]
where $L_i$ is the tableau having $\boxtimes$ at the left end of the $i$th row of $\gamma^c/\alpha$. Likewise, 
\[\LRby = \{R_1, \ldots, R_r\},\]
where $R_i$ is the tableau having $\boxtimes$ at the right end of the $i$th row of $\gamma^c/\alpha$.

We have the following:
\begin{align*}
\esh(L_i) &= R_{i+1} \pmod{r}
\end{align*}
We will say that $\esh(L_r) = R_1$ is a \newword{special jump}, and any other application of $\esh$ to $L_i$ for $i\neq 1$ is \newword{non-special}.

\end{enumerate}
\end{theorem}

\begin{proof}
  Part 1 is clear because $\esh$ is a bijection between two one-element sets.
  
  For Part 2, it is clear that these are the only fillings.  So, it suffices to show that $\esh(L_i)=R_{i+1}$ for any $i$, where the indices are taken modulo $r$.  We will show this by induction on the size of $\alpha$.
  
  For the base case, if $\alpha=\emptyset$, then since we are in the case of Part 2, the total shape of the $\ybox$ and the tableau is a single row of length $m+1$.  (The other possibility is that the total partition shape is $(m,1)$, and the $\ybox$ slides up and down between the two squares of the first column, which is in Case 1.)  So, $\LRyb$ and $\LRby$ both have one element, $L_1$ and $R_1$ respectively, and so $L_1$ must go to $R_1$ under $\esh$ and we are done.  Notice that this base case is a special jump.
  
  Now, suppose the theorem holds for a given $\alpha$, and we wish to show it holds for a partition $\alpha'$ formed by adding an outer co-corner to $\alpha$.  Let $T'\in \LR(\alpha',\ebox,\beta,\gamma')$ for some $\beta$ and $\gamma'$, and let $T\in \LRyb$ be the tableau formed by the first jeu de taquin slide on $T'$ in the rectification of $(\ybox,T')$ in the evacuation-shuffle, where we start with the unique outer co-corner of $\alpha$ that is contained in $\alpha'$. Here $\gamma$ is formed from $\gamma'$ by adding the unique corner determined by this slide.  
  
  Note that $T'=L_i'$ for some $i$, where $L_i'$ is the tableau having the $\ybox$ in the $i$th row from the top in $(\gamma')^c/\alpha'$.  Defining $R_i'$ similarly, we wish to show $\esh(L_i')=R_{i+1}'$ with the indices mod $r$.

  Recall that $\esh$ is the procedure of rectifying $(\ybox,T')$, shuffling the box past the rectified tableau, and then unrectifying both using the reverse sequence of slides.  Let $S\in \LRby$ be the tableau preceding the last unrectification step in forming $S'=\esh(T')$.  These steps necessarily involve the same inner and outer co-corners, and so $S$ and $T$ have the same shape.  Furthermore, $\esh(T)=S$ by the definition of $\esh$, and so by the induction hypothesis $S$ is formed from $T$ by one of the two Pieri rules.  
  
  We will use this to show that $S'$ is formed from $T'$ by the same rules, by considering the rectification/unrectification step that relates them to $S$ and $T$ respectively.  We consider the cases of a special jump and a non-special jump separately. Let $r$ be the number of nonempty rows of $(\gamma')^c/\alpha'$.
  
  \textbf{Case 1:}  Suppose $T'=L_i'$ for some $i\neq r$.  The tableau $T$ is formed by a single inwards jeu de taquin slide on $T'$, which can either be on the inner co-corner just to the left of the $\ybox$ or not. 
  
  If the inner co-corner we start at is to the left of the $\ybox$ in $T'$, then since we assumed our shape has no vertical domino, the entire row containing the $\ybox$, say row $r$, simply slides to the left to form $T$.  Then by the induction hypothesis, $S$ has the $\ybox$ at the end of the next row down, either just below the $\ybox$ in $T$ (the vertical domino case) or to its left.  Clearly $S'$ is formed by sliding the new contents of row $r$ back to the right, and so $S'=R'_{i+1}$ as desired.
  
  Otherwise, if the inner co-corner we start at is not to the left of the $\ybox$ in $T'$, the inwards slide consists of either (a) sliding a horizontal row of $1$'s to the left, if the co-corner is to the left of but not above a $1$, (b) sliding a $1$ on an outer corner up by one row, if the co-corner is just above this $1$.  
  
  In the subcase (a), the number of rows remains unchanged and $T=L_i$.  Thus $S=R_{i+1}'$ by the induction hypothesis and we are done.  For (b), the number of rows either remains the same and we are done again, or the $1$ that we moved up forms a new row.  If the new row is above the $\ybox$, then $T=L_{i+1}$, by the induction hypothesis $S=R_{i+2}$, and $S'$ is formed by moving the $1$ back down and therefore $S'=R_{i+1}'$, as desired.  Otherwise, if the new row is below the $\ybox$, we have $T=L_i$ and $S=R_{i+1}$, keeping in mind that if the $\ybox$ is in row $i$ in $T$ then row $i+1$ is the new row and the $\ybox$ is in this new square in $S$.  Therefore we again have $S'=R_{i+1}'$, and we are done.
  
  \textbf{Case 2:}  Suppose $T'=L_r'$.  Then the $\ybox$ is weakly below and strictly to the left of all other entries. Notice that any inwards jeu de taquin slide does not change this property; hence $T=L_q$ where $q$ is the bottom row of $T$.  Then $S=R_1$ by the induction hypothesis, and by the same argument, any outwards jeu de taquin slide doesn't change the property of the $\ybox$ being weakly above and strictly to the right of the rest of the entries in $S$.  Hence $S'=R'_1$, as desired.  \qed
\end{proof}

For use in Section \ref{sec:main-result}, we describe how to determine the outcome of the Pieri case based on the location of the $\ybox$ in \emph{either} the original skew tableau \emph{or} its rectification:
\begin{proposition} \label{prop:pieri-criteria}
Let $\beta=(m)$. The following are equivalent for $T \in \LRyb$:
\begin{itemize}
\item[(i)] Applying $\esh$ results in a special jump;
\item[(ii)] The $\ybox$ precedes the rest of $T$ in reading order;
\item[(iii)] The rectification of $T$, including the $\ybox$, forms a horizontal strip.
\end{itemize}
\end{proposition}
\begin{proof}
This follows immediately from the proof of the Pieri case. \qed 
\end{proof}

\subsection{The proof of Theorem \ref{thm:main-theorem}}

% We now prove Theorem \ref{thm:main-theorem}. \\

%\begin{proof}[Proof of Theorem \ref{thm:main-theorem}]{\ }\\
\noindent{\bf Step 1}. Fix $(\ybox,T) \in \LRby$. We choose $s = s(\ybox,T)$ as follows: consider $\sh_1 (\sh_2 \sh_1(\ybox,T))$, the tableau obtained by rectifying, then shuffling $\ybox$ past $T$. Let $s$ be the row containing $\ybox$. We will use the $s$-decomposition with this choice of $s$, and compute the effect of $e_t \cdots e_1$ on $(\ybox,\iota_s(T))$. We write
\[\iota_s(T) = (H_1,\ldots, H_{s-1}, V_s, \ldots, V_t).\]

We note that, if we rectify and shuffle the $\ybox$ past the entirety of $\iota_s(T)$, the shuffle path of the $\ybox$ through the rectification of $\iota_s(T)$ is to move (one square at a time) down to row $s$, then over to the right. (See Figure \ref{fig:R-diagram}.)\\

\noindent{\bf Step 2}. We show that $\esh$ and $\lesh$ agree up to Phase 1.

\begin{lemma} \label{lem:phase1-agrees} Suppose $s > 1$ and let $1 \leq i \leq s-1$. Then $T_i = e_i \cdots e_1(\ybox,\iota_s(T))$ agrees with the result of applying $i$ Phase 1 local evacuation-shuffle moves to $(\ybox,T)$.
\end{lemma}
\begin{proof}
Assume the statement holds for $i-1$ (this is vacuous for $i=1$) and write
\[T_{i-1} = e_{i-1} \cdots e_1(\ybox,\iota_s(T)) = (H'_1, \ldots, H'_{i-1}, \ybox, H_i, \ldots, H_{s-1}, V_s, \ldots, V_t.\] In $T_i$, the $\ybox$ lies between $H_{i-1}$ and $H_i$.

We compute $e_i(T_{i-1})$. For simplicity, let $H'$ be the concatenation of $H'_1, \ldots, H'_{i-1}$. We are effectively computing
\[\xymatrix{
(T_\alpha,H',\ybox,H_i, \cdots) \ar@{|->}[d]_-{\sh_3\sh_2\sh_1} \\
(H'',\ybox,H_i',T'_\alpha, \cdots) \ar@{|->}[d]_-{\sh_2} \\
(H'',H''_i,\ybox,T'_\alpha, \cdots) \ar@{|->}[d]_-{\sh_1\sh_2\sh_3} \\ (T_\alpha,H''',H'''_i,\ybox, \cdots). }\]
By our definition of $s$, in the partial rectification $H'' \sqcup \ybox \sqcup H'_i$, the $\sh_2$ step causes the box to move down to row $i+1$ (since $i \leq s-1$).

In particular, we see that $S = \ybox \sqcup H'_i$ forms a straight shape in the partial rectification. Shuffling $\ybox$ and $H'_i$ does not change the overall (trivial) dual equivalence class of $S$; consequently, $e_i$ has no effect on the dual equivalence classes of $T_i$ other than the individual classes of $\ybox$ and $H_i$.

Moreover, since $e_i$ rectifies $(\ybox,H_i)$ to a straight shape, then shuffles and un-rectifies, it must have the same effect as simply applying $\esh$ to the pair $(\ybox,H_i)$, i.e. the Pieri Case. Moreover, in the rectification, the $\ybox$ shuffled downward rather than right, so by Proposition \ref{prop:pieri-criteria}, we see that \emph{prior} to rectifying, there was an $i$ below the $\ybox$. Thus we are in the non-special Pieri case, which agrees with the $i$-th (Phase 1) step of $\lesh$. \qed
\end{proof}

\begin{lemma}
The transition step of $\lesh(\ybox,T)$ is $s$.
\end{lemma}
\begin{proof}
By a similar argument to the previous lemma, we see that, had we used the $(s+1)$-decomposition rather than the $s$-decomposition, then applying $e_s$ would, after rectifying, slide the $\ybox$ all the way to the right through the $s$-th row. This is the `special jump' of the Pieri case (which would not agree with the behavior of $\lesh$). By Proposition \ref{prop:pieri-criteria}, this occurs only when, \emph{prior} to rectifying, there are no $s$'s below the $\ybox$. This is precisely the condition for $\lesh$ to transition at step $s$. \qed
\end{proof}

\noindent{\bf Step 3}. We have shown that $\lesh$ and $\esh$ agree up to the transition point of $\lesh$, and that this corresponds to the bend in the shuffle path of the $\ybox$ through the rectification of $\iota_s(T)$. We are left with determining the effects of $e_s, \ldots, e_t$.

\begin{lemma}[Antidiagonal symmetry] \label{lem:antidiag}
For $i \geq s$, $e_i$ corresponds to a \emph{conjugate move} across the vertical strip $V_i$, as in Definition \ref{def:conjugate-pieri}.
\end{lemma}
\begin{proof}
We will prove this for all remaining steps simultaneously. Put $M = e_{s-1} \cdots e_1(T)$. We have
\[M = (T_\alpha, H'_1, \ldots, H'_{s-1},\ybox,V_s, \ldots, V_t, T_\gamma^R),\]
where each $H'_i$ is a horizontal strip and each $V_i$ a vertical strip. Let $M_H, M_V$ be the concatenations of the $H'_i$'s and of the $V_j$'s. (We note that $M_H$ is the union of the first $s-1$ strips of $T$ at the transition point of $\lesh$, and that $M_V$ is simply the rest of the tableau.)

The remainder of the computation corresponds to partial-evacuation-shuffling the $\ybox$ past $M_V$,
\[\esh(T) = \sh_1\sh_2\sh_3\circ \sh_2 \circ \sh_3\sh_2\sh_1(T_\alpha, M_H, \ybox, M_V, T_\gamma^R).\]

We know that, in $\sh_3\sh_2\sh_1(M)$, the $\ybox$ and $M_V$ class form a straight shape. Thus, by similar reasoning to the proof of Lemma \ref{lem:phase1-agrees}, the remaining computation is the same as \emph{ordinary}  -- not partial -- evacuation-shuffling the pair $(\ybox,M_V)$. %, that is, it computes
%\[\esh(T_{\tilde \alpha},\ybox,M_V,T_\gamma^R) \in \LR(\tilde{\alpha}, \ebox, \tilde{\beta}, \gamma),\]
%where $T_{\tilde{\alpha}}$ is the concatenation $T_\alpha \sqcup M_H$, of shape $\tilde \alpha$, and $\tilde{\beta}$ is $\beta$ with its first $s-1$ rows deleted.
Note that in this (smaller) computation, the $\ybox$ slides right after rectifying, i.e. $\lesh(\ybox,M_V)$ begins in Phase 2, so our earlier results do not apply.
However, by Lemma \ref{lem:outer-esh}, we may instead write 
\[\esh(\ybox,M_V) = \sh_3\sh_4\circ \sh_3 \circ \sh_4\sh_3(\ybox,M_V),\]
i.e. we may instead anti-rectify outwards, %past $T_\gamma^R$,
then shuffle and return. To simplify the situation, we `rotate and transpose', obtaining
\[(M_V',\ybox) = \LR\big( \DE\big(\ybox,M_V\big){}^{R*} \big),\]
%and (since the $\ybox$ is now on the \emph{outside}) we wish to show
%\[\esh(M_V',\ybox) = \lesh^{-1}(M_V',\ybox),\]
%i.e. we wish to `work backwards' through the algorithm. 
%
%\[M' = (T_{\gamma^\ast},M_V',\ybox,T_{{\tilde \alpha}^\ast}^R) \in \LR(\gamma^\ast,\tilde{\beta}^\ast,\ybox,\tilde{\alpha}^\ast).\]
%[Note from Jake: this line is quite a mouthful of notation, and it might look like it makes things worse, not better. But it was *painful* to articulate the remainder of the explanation without actually reflecting $M$, since I had to use the phrases like ``anti-" everywhere, including when applying the earlier results on Phase 1. This should be improved...] \\
%
Note that the vertical strips of $M_V$ correspond to the horizontal strips of $M_V'$ after this transformation. That is, $M_V'$ has entries $i$ in the squares of the antidiagonal reflection of the strip $V_{t+1-i}$. In the rectification of $(\ybox,M_V)$, the $\ybox$ was to the left of one square from each $V_i$. So the anti-rectification of $M$ has the $\ybox$ in the leftmost corner, and so (by reflecting over the antidiagonal) the rectification of $(M_V',\ybox)$ has the $\ybox$ at the bottom of the first column:

\begin{figure}[h]
\begin{center} \hspace{0.2cm}
\includegraphics[scale=0.75]{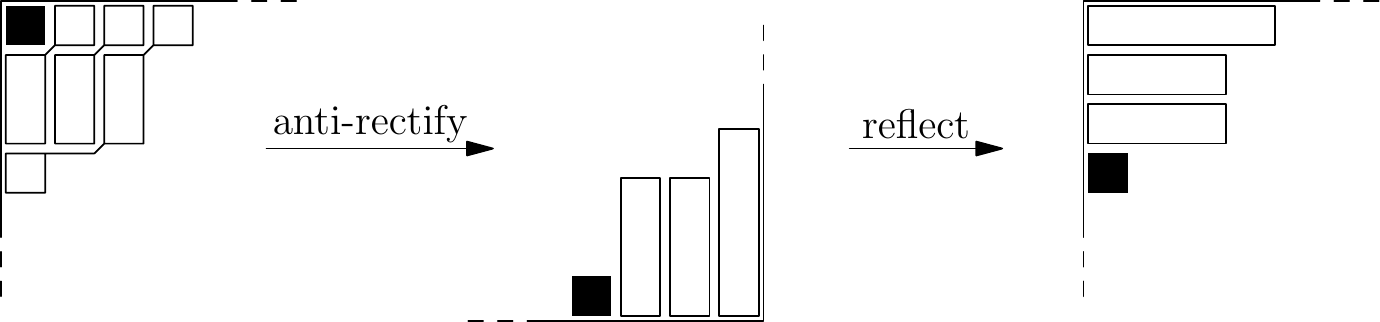}
\caption{From left to right, the rectification and anti-rectification of $(\ybox,M_V)$ and the rectification of $(M'_V,\ybox)$.}\vspace{-0.6cm}
\end{center}
\end{figure}
We set $(\ybox,N') = \esh(M_V',\ybox)$ and we compare $\esh(\ybox,N')$ to $\lesh(\ybox,N')$. By the above observation, $\esh(\ybox,N')$ corresponds to a local evacuation-shuffle that stays entirely in Phase 1.
Thus, by our existing Lemmas on evacuation-shuffling in Phase 1, we see that the \emph{partial} evacuation-shuffles of $\ybox$ through $N'$ correspond to non-special Phase 1 moves applied to the skew tableau. Reflecting back to our original setting $(\ybox,M_V)$, we deduce that the remaining Phase 2 \emph{partial} evacuation shuffles result in successive non-special conjugate moves of the $\ybox$ through the strips $V_s, \ldots, V_t$.  \qed 
\end{proof}

\noindent{\bf Step 4}. Finally, we prove that the description of $e_s, \ldots, e_t$ corresponding to conjugate Pieri moves produces the same $\ybox$ movements as Phase 2 of local evacuation-shuffling. Note that this step involves only ballot tableaux, not dual equivalence classes.

\begin{lemma} Conjugate moves correspond to nontrivial movements of the $\ybox$, in Phase 2, through its evacu-shuffle path.
\end{lemma}

\begin{proof}

First, notice that the Phase 2 algorithm, as described in Remark \ref{rmk:alternate-phase-2}, can be stated as follows.  Starting with $i=s$, at each step choose the smallest $k\ge i$ for which the $(k,k+1)$ suffix of the $\ybox$ is \textit{not} tied, and then switch the $\ybox$ with the first $k$ that occurs after it in reading order, incrementing $i$ to $k+1$ and repeating.  We will show that shuffling past the $V_j$'s using conjugate moves does the same thing.

 Suppose we are moving the $\ybox$ across the strip $V_{j+1}$.  Then either on the previous move it switched places with an element $i$ in $V_j$, or $i=s$ and it is at the start of Phase 2.  We first show that the $\ybox$ switches with an element $k\ge i$ by considering these two cases separately.
 
 \textbf{Case 1.}  If $i=s$ and it is the start of Phase 2, the next move will switch with an element $k\ge i=s$ because the vertical strips all have entries of size at least $s$.
 
  \textbf{Case 2.}  If the $\ybox$ just switched with an $i$ in $V_j$, then there exists an $i$ in $V_{j+1}$ because the $V_j$'s weakly increase in length as $j$ increases by the definition of $s$-decomposition.  Furthermore, the $i$ in $V_{j+1}$ occurs after the $i$ in $V_j$ in reading order, so the $\ybox$ switches with an entry $k$ in $V_j$ which weakly precedes the $i$ in reading order.  We must therefore have $k\ge i$ since $V_j$'s entries increase down the strip.

   Finally, in either case, suppose $k'$ is an index with $i\le k'<k$.  Then the $k'$ and $k'+1$ in $V_{j+1}$ both occur later than $\ybox$ in the reading word before the move, and by the definition of $s$-decomposition this means that the $k'$ and $k'+1$ in each later strip $V_{j'}$ also occur after the $\ybox$.  Hence the $(k',k'+1)$ suffix of $\ybox$ is tied prior to the move.  Since the $k+1$ in $V_{j+1}$ precedes the $\ybox$, the $(k,k+1)$ suffix is not tied.
   
   So indeed, the $\ybox$ switches with the smallest $k\ge i$ for which the $(k,k+1)$ suffix is tied. \qed 
\end{proof}

This completes the proof of Theorem \ref{thm:main-theorem}.

\begin{remark}
Since we work with semistandard tableaux, a natural question is to ask what happens if we use only horizontal strips to factor $\esh$ (i.e. if we attempt to use the $(\ell(\beta)+1)$-decomposition for all $T$) rather than the appropriate $s$-decomposition. In fact, an earlier version of our algorithm used this approach; its proof relied on the `un-rectification' method, as demonstrated for the Pieri Case. However, the proof is more difficult, and there are drawbacks to the resulting local algorithm. The most notable are that it does not preserve ballotness at the intermediate steps, and, after Phase 1, it consists of 3-cycles rather than simple transpositions switching the $\ybox$ with one other entry. These drawbacks make the applications to K-theory (see Section \ref{sec:K-theory}) harder or impossible to deduce.
\end{remark}

\subsection{Corollaries on Evacuation-Shuffling}\label{sec:corollaries-evacu-shuffling}
%%%%%%%%%%%%%%%%%%%%%%%%%%%%%
%%%%%%%% Corollaries %%%%%%%%
%%%%%%%%%%%%%%%%%%%%%%%%%%%%%

For each of the following corollaries, let $(\ybox,T) \in \LRyb$.

\begin{corollary}\label{cor:transition-step}
Suppose the transition step of $\lesh(\ybox,T)$ is $s$. Then $\beta$ has an outer co-corner in row $s$, and the evacu-shuffle path of the $\ybox$ has length exactly $s+\beta_s$, including the initial and final locations of the $\ybox$.
\end{corollary}
\begin{proof}
From the proof of Theorem \ref{thm:main-theorem}, the $\ybox$ ends up in the square $(s,\beta_s+1)$ after rectifying and shuffling past $T$. Thus, this square is an outer co-corner of $\beta$.

From the local description of $\esh$, the box moves through a total of $s-1$ squares in Phase 1. From the description of Phase 2 in terms of conjugate moves, the $\ybox$ moves through $\beta_s$ squares in Phase 2. \qed 
\end{proof}

\begin{corollary}[Antidiagonal symmetry and evacu-shuffle paths] \label{cor:antidiag-evacu-path}
Define $(T^{R\ast},\ybox)$ by rotating and transposing $(\ybox,\DE(T))$, then taking its highest-weight representative. %This is a tableau in $\LR_\eset^\rect(\gamma^\ast,\beta^\ast,\ybox,\alpha^\ast)$.

Similarly, for $(S,\ybox) \in \LRby$, define $(\ybox,S^{R\ast})$ the same way. Then:
\[\lesh(\ybox,T) = (S,\ybox) \qquad \text{ iff } \qquad (T^{R\ast},\ybox) = \lesh(\ybox,S^{R\ast}).\]
%that is, ``reflecting over the antidiagonal" exchanges $\lesh$ with $\lesh^{-1}$.

Moreover, the evacu-shuffle path of the $\ybox$ for $\lesh(\ybox,S^{R\ast})$ is the antidiagonal reflection of the evacu-shuffle path of the $\ybox$ for $\lesh(\ybox,T)$.
\end{corollary}
See Figure \ref{fig:antidiag-evacu-path} for an example of this phenomenon.
\begin{proof}
As a map on dual equivalence classes, $\esh$ is its own inverse, and it commutes with transposing and rotating (since shuffling does). However, since the $\ybox$ is then on the opposite side of the tableau, $\esh$ corresponds by our main theorem to $\lesh^{-1}$. Thus $\lesh(\ybox,S^{R\ast}) = (T^{R\ast},\ybox)$.

To see that the evacu-shuffle paths are the same, we compare $s$-decompositions. Observe that rotating and transposing interchanges the functions $\iota_H$ and $\iota_V^\ast$ of Lemma \ref{lem:factor-other}. So, the  $s$-decomposition of $T^{R\ast}$ corresponds to a `dual' $s$-decomposition of $T$,
\[\iota^\ast_s(T) = (\iota_H)^{s-1} \circ (\iota_V^\ast)^{\beta_s}(T) \ \big(= \iota_s(T^{R\ast})^{R\ast}\big),\]
where we extract the $\beta_s$ vertical strips first, \emph{then} extract the $s-1$ horizontal strips.

Consider rectifying and shuffling $(\ybox,T)$. It is easy to see that the shuffle path is the same for both $\iota_s(T)$ and $\iota_s^\ast(T)$. The proof of Theorem \ref{thm:main-theorem} then shows that the partial evacuation-shuffles corresponding to $\iota_s^\ast$ give the same Pieri and conjugate-Pieri moves as those corresponding to $\iota_s$. \qed 
\end{proof}

\begin{corollary}
The following are equivalent:
\begin{itemize}
\item[(i)] The transition step of $\lesh(\ybox,T)$ is $s$.
\item[(ii)] Let $(\ybox,T')$ be the `transposed class', obtained by transposing $(\ybox,\DE(T))$, then taking the highest-weight representative. Then the transition step of $\lesh$ on $(\ybox,T')$ is $\beta_s+1$.
\item[(iii)] Let $(T'',\ybox)$ be the `rotated class', obtained by rotating $(\ybox,\DE(T))$, then taking the highest-weight representative. Then the transition step of $\lesh^{-1}$ on $(T'',\ybox)$ is $s$.
\end{itemize}
\end{corollary}
\begin{proof}
To see that (i) implies (ii), note that shuffling commutes with transposing dual equivalence classes, so in Step 1 of the proof of Theorem \ref{thm:main-theorem}, we find that the $\ybox$ is in the square $(\beta_s+1,s)$ after rectifying and shuffling. This means the transition step of $(\ybox,T')$ will be $\beta_s+1$. The same reasoning with $T$ and $T'$ exchanged shows (ii) implies (i).

To see that (ii) implies (iii), we use the previous corollary. Transposing \emph{and} rotating exchanges the Phase 1 and Phase 2 portions of the evacu-shuffle path. But the length of the Phase 1 portion of the path is exactly the value of the transition step. As above, the same reasoning with $(\ybox,T)$ and $(T'',\ybox)$ exchanged shows (iii) implies (ii).  \qed 
\end{proof}

Finally, we briefly consider the running time of $\lesh$. We assume the set $\LRyb$ is given along with, for each $(\ybox, T)$, the $1$-decomposition of $T$ into vertical strips. (Computing this decomposition in advance does not increase the asymptotic running time of computing $\LRyb$, since it can be obtained by simply labeling each $i$ with its distance from the end of its horizontal strip as the tableau is generated.)
\begin{corollary}\label{cor:running-time}
Given $\LRyb$ as above, computing $\lesh$ takes $O(b)$ steps, where $b = \ell(\beta) + \ell(\beta^\ast)$. Computing the entire orbit decomposition of $\omega$ takes $O(b \cdot |\LRyb|)$ steps.
\end{corollary}
\begin{proof}
We compute $\lesh(\ybox,T)$ directly, for any transition step $s$, updating the $1$-decomposition at the same time. Note that during a Phase 1 move, the $i$ that switches with $\ybox$ remains part of the same \emph{vertical} strip, since its position among the $i$'s in reading order is unchanged. Thus, we apply Phase 1 moves until the transition step, updating the vertical strips accordingly.

For Phase 2, note that the $s$-decomposition is simply the $1$-decomposition with all squares of value less than $s$ deleted. We may thus compute conjugate Pieri moves using the 1-decomposition.

Note that there are at most $\ell(\beta) + \ell(\beta^\ast)$ steps in all.  \qed 
\end{proof}
\begin{figure}[t]
\begin{center}
\begin{tabular}{m{1cm} @{$=\ $} m{3.5cm} m{2.5cm} @{$=\ $} m{3cm}}
\raggedright $\ \ \ \ T $ & \centering \includegraphics[scale=0.85]{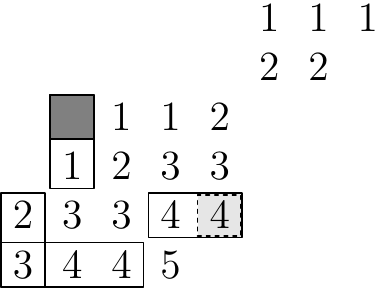} & $\ \xrightarrow{\lesh} \ \ \ S $ & \includegraphics[scale=0.85]{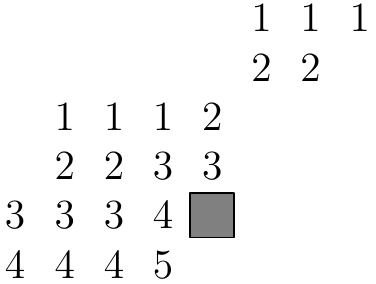} \vspace{0.1cm} \\
\raggedright $\ T^{R\ast} $ & \centering \includegraphics[scale=0.85]{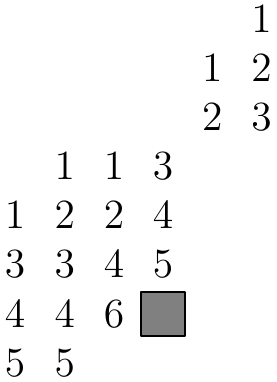} & $\ \xleftarrow{\lesh}\ \ S^{R\ast} $ & \includegraphics[scale=0.85]{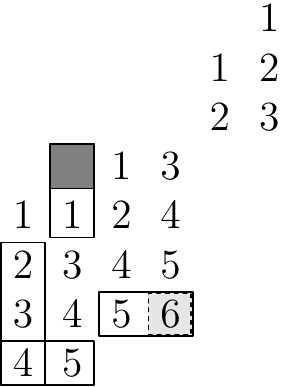}
\end{tabular}
\end{center}
\caption{An example of antidiagonal reflection. The dual equivalence classes of (the standardizations of) $T$ and $T^{R\ast}$ are antidiagonal reflections of one another, as are those of $S$ and $S^{R\ast}$. By Corollary \ref{cor:antidiag-evacu-path}, their evacu-shuffle paths are likewise antidiagonally-reflected. }
\label{fig:antidiag-evacu-path}
\end{figure}

\section{Connections to K-theory}\label{sec:K-theory}

\subsection{Background on K-theoretic (genomic) tableaux}

We recall the results we need on increasing tableaux and K-theory. The structure sheaves $\mathcal{O}_\lambda$ of Schubert varieties in $Gr(k,\mathbb{C}^n)$ form an additive basis for the K-theory ring $K(Gr(k,\mathbb{C}^n))$, and they have a product formula
\[[\mathcal{O}_\mu] \cdot [\mathcal{O}_\nu] = \sum_{|\lambda| \geq |\mu| + |\nu|} (-1)^{|\lambda| - |\mu| - |\nu|}k_{\mu \nu}^\lambda [\mathcal{O}_\lambda],\]
for certain nonnegative integer coefficients $k_{\mu \nu}^\lambda$. These coefficients enumerate certain tableaux, which we now discuss.

In \cite{bib:ThomasYong}, Thomas and Yong have defined a K-theoretic jeu de taquin for \newword{increasing tableaux}, i.e. tableaux that are both row- and column-strict; the tableaux analogous to highest-weight standard tableaux are those whose K-rectification is superstandard. When the K-rectification shape is a single row $\beta = (d)$, these are the \emph{Pieri strips of max-content $d$}:
\begin{definition}[\cite{bib:ThomasYong}, Section 5]
Let $\lambda/\mu$ be a horizontal strip, (no two squares are in the same column). We say a tableau $T$ of shape $\lambda/\mu$ is a \newword{Pieri strip} if:
\begin{itemize}
\item[(1)] the rows of $T$ are strictly increasing, 
\item[(2)] the reading word of $T$ is weakly increasing and does not omit any value $1, \ldots, \max(T)$.
\end{itemize}
We say the \newword{max-content} of $T$ is $\max(T)$.
\end{definition}

\begin{example}
For the shape $\lambda/\mu = {\tiny \young(:::\hfil\hfil,:\hfil\hfil,\hfil)}$, there is one Pieri strip of max-content $5$, two of max-content $4$ and one of max-content $3$. These are, respectively:
\[\young(:::45,:23,1) \qquad \young(:::34,:12,1) \qquad \young(:::34,:23,1) \qquad \young(:::23,:12,1).\]
\end{example}

For general shapes, there is an analogous theory of `(ballot) semistandard increasing tableaux'. These are the \newword{genomic tableaux} defined by Pechenik in \cite{bib:Pechenik}, whose entries are subscripted integers $i_j$, which we now define.
\begin{definition}[\cite{bib:Pechenik}]
Let $T$ be a genomic tableau with entries $i_j$. We call $i$ the \emph{gene family} and $j$ the \emph{gene}. First, we say $T$ is \newword{semistandard} if:
\begin{itemize}
\item The tableau $T_{ss}$ obtained by forgetting the genes is semistandard (that is, each gene family forms a horizontal strip);
\item Within each gene family, the genes form a Pieri strip.
\end{itemize}
We say the \newword{$K$-theoretic content} of $T$ is $(c_1, \ldots, c_r)$ if $c_i$ is the max-content of the Pieri strip of genes in the $i$-th gene family. Finally, we say $T$ is \newword{ballot} if it is semistandard and has the following property:
\begin{itemize}
\item[$(*)$] Let $T'$ be any genomic tableau obtained by deleting, within each gene family of $T$, all but one of every repeated gene. Let $T'_{ss}$ be the tableau obtained by deleting the corresponding entries of $T_{ss}$. Then the reading word of $T'_{ss}$ is ballot.
\end{itemize}
\end{definition}
\noindent We write $K(\lambda/\mu;\nu)$ for the set of ballot genomic tableaux of shape $\lambda/\mu$ and $K$-theoretic content $\nu$.
\begin{theorem}[\cite{bib:Pechenik}]
We have $k_{\mu  \nu}^\lambda = |K(\lambda/\mu;\nu)|$.
\end{theorem}

We are most concerned with the case of partitions $\alpha, \beta, \gamma$ with $|\alpha| + |\beta| + |\gamma| = k(n-k)-1$. In this case there will only be one repeated gene, in one gene family. Let $K(\gamma^c/\alpha;\beta)(i)$ be the set of increasing tableaux in which $i$ is the repeated gene family. For convenience, we state the following simpler characterization of this set:
\begin{lemma} \label{lem:genomic-criterion}
Let $T$ be an (ordinary) semistandard tableau of shape $\gamma^c/\alpha$ and content equal to $\beta$ except for a single extra $i$. Let $\{\ybox_1, \ybox_2\}$ be two squares of $T$, such that
\begin{itemize}
\item[(i)] The squares are non-adjacent and contain $i$,
\item[(ii)] There are no $i$'s between $\ybox_1$ and $\ybox_2$ in the reading word of $T$,
\item[(iii)] For $k = 1,2,$ the word obtained by deleting $\ybox_k$ from the reading word of $T$ is ballot.
\end{itemize}
There is a unique ballot genomic tableau $T' \in K(\gamma^c/\alpha; \beta)(i)$ corresponding to the data $(T,\{\ybox_1,\ybox_2\})$. Conversely, each $T'$ corresponds to a unique $(T,\{\ybox_1,\ybox_2\})$.
\end{lemma}
\begin{proof}
The gene families of $T'$ are the entries of $T$. For $j \ne i$, the $j$-th gene family of $T'$ has all distinct genes, obtained by standardizing the $j$-th horizontal strip of $T$. For the $i$-th gene family, there are exactly two repeated genes, in the squares $\ybox_1, \ybox_2$. This uniquely determines the Pieri strip. Ballotness of $T'$ is then equivalent to (iii).  \qed 
\end{proof}

\subsection{Generating genomic tableaux}
We now establish connections between local evacuation-shuffling and genomic tableaux. We first describe how the tableaux $K(\gamma^c/\alpha;\beta)$ arise from evacuation-shuffling. In fact, each tableau arises once during some step of Phase 1 and once during Phase 2, for some $T_1, T_2 \in \LRyb$.

Let $(\ybox,T) \in \LRyb$. Suppose the evacu-shuffle path for $\lesh(\ybox,T)$ is not connected. This occurs whenever $\lesh$ applies a Pieri or conjugate Pieri move. Let $\ybox_1, \ybox_2$ be two successive non-adjacent squares in the path, and suppose the $\ybox$ switched with an $i$ during this move (that is, the movement occurred during $\pieri_i$ or $\jump_i$). Let $T_i$ be the tableau before this step, with the $\ybox$ replaced by $i$. We will show using Lemma \ref{lem:genomic-criterion} that $(T_i,\{\ybox_1, \ybox_2\})$ corresponds to a genomic tableau. See Figure \ref{fig:genomic-bijection} for an example.

\begin{figure}[b]
\[\left(\ \raisebox{-.45\height}{\includegraphics[scale=0.7]{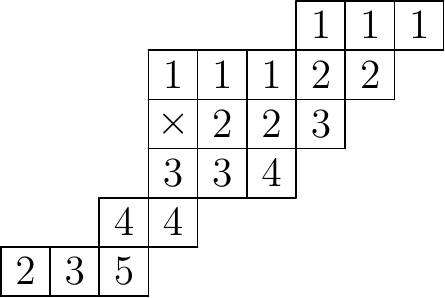}}
\xrightarrow{\pieri_2}
\raisebox{-.45\height}{\includegraphics[scale=0.7]{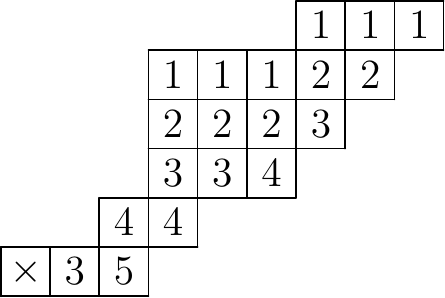}}\ \right)\
\xmapsto{\ \varphi_1\ }
\raisebox{-.45\height}{\includegraphics[scale=0.7]{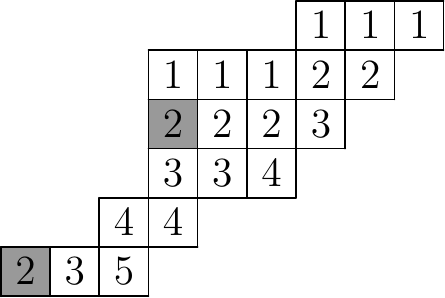}}\]
\caption{A Pieri move, and the genomic tableau generated by $\varphi_1$. The shaded squares of the genomic tableau correspond to $\ybox_1$ and $\ybox_2$, that is, the unique repeated gene.}
\label{fig:genomic-bijection}
\end{figure}

\begin{theorem} \label{thm:generating-ktheory}
The data $(T_i,\ybox_1,\ybox_2)$ corresponds to a ballot genomic tableau, as in Lemma \ref{lem:genomic-criterion}. Moreover, as $T$ ranges over $\LRyb$, every tableau $T_K \in K(\gamma^c/\alpha;\beta)(i)$ arises exactly once this way in Phase 1 and once more in Phase 2. This gives two bijections:

\begin{center}
\begin{tabular}{l c l}
$\varphi_1$ : & $\big\{\pieri_i \text{ moves} \big\}$ & $\to\ K(\gamma^c/\alpha;\beta)(i),$ \\
$\varphi_2$ : & $\big\{\cpieri_j \text{ moves during } \jump_i \big\}$ & $\to\ K(\gamma^c/\alpha;\beta)(i)$.
\end{tabular}
\end{center}

\end{theorem}
\begin{proof}
By construction, the squares are non-adjacent. From the definition of local evacuation-shuffling, there is no $i$ between $\ybox_1$ and $\ybox_2$ in the reading word of $T$. We need only check that after deleting either one of $\ybox_1, \ybox_2$, the reading word of $T_i$ is ballot. This follows from Theorem \ref{thm:ballotness}.

We show that $\varphi_1$ is bijective. It is clearly injective. Next, given a genomic tableau $T_K$, let $(T,\{\ybox_1,\ybox_2\})$ be as defined in Lemma \ref{lem:genomic-criterion}. Replace either $\ybox$ entry with $i$, and leave the other as $\ybox$. This gives a pair of tableaux $T',T''$, which differ by an ordinary, non-vertical Pieri move. An argument similar to that of Theorem \ref{thm:ballotness} then shows that applying Reverse Phase 1 moves yields an element $T \in \LRyb$. (It is important that \emph{both} tableaux $T',T''$ are ballot.)

The proof for $\varphi_2$ is identical, only we elect to think of $T', T''$ as differing by a movement of the $\ybox$ in Phase 2. We again work backward to get to $T$. Note that the tableaux $T', T''$ occur in the opposite order when we think of them as arising during Phase 2.  \qed 
\end{proof}

\begin{example}[Pieri case, revisited] \label{exa:ktheory-pieri}
Suppose $\beta$ has only one row, and $\gamma^c/\alpha$ is a horizontal strip with $r$ nonempty rows. With notation as in Theorem \ref{thm:Pieri}, we have
\[\LRyb = \{L_1, \ldots, L_r\}, \qquad \omega(L_i) = L_{i+1 \pmod r}.\]
In this case, the corresponding genomic tableaux are the Pieri strips on $\gamma^c/\alpha$ of $K$-theoretic content $\beta$. Let $G_{i,i+1}$ be the tableau in which the two equal entries are at the beginning of the $i$-th row and the end of the $(i+1)$-st.

In Phase 1, the ordinary step $\omega(L_i) = L_{i+1}$ generates $G_{i,i+1}$ (for $1 \leq i < r$), while the special jump does not correspond to a genomic tableau.

In Phase 2, the ordinary steps $\omega(L_i) = L_{i+1}$ do not correspond to genomic tableaux, while the special jump generates all of them at once.
\end{example}

\subsection{The sign and reflection length of \texorpdfstring{$\omega$}{w} via genomic tableaux}

We recall the statements about $\omega$ known from geometry:
\begin{align}
|\Kabg| &\geq \mathrm{rlength}(\omega), \label{eqn:recall-ineq} \\
|\Kabg| &\equiv \mathrm{sgn}(\omega) \pmod 2. \label{eqn:recall-parity}
\end{align}
where $\mathrm{sgn}(\omega)$ is $0$ or $1$ when $\omega$ is even or odd respectively, and $\mathrm{rlength}(\omega)$ denotes the \textit{reflection length}, the minimum length of a factorization of $\omega$ as a product of transpositions (permutations consisting of a single $2$-cycle).  Note that the right-hand sides of equations \eqref{eqn:recall-ineq} and \eqref{eqn:recall-parity} are the same, mod $2$.

We now give enumerative proofs of these statements, using the bijection $\varphi_1$ of Theorem \ref{thm:generating-ktheory} to count ballot genomic tableaux.  The key idea is to break down the steps of the local algorithm and thereby factor $\omega$ into simpler permutations.

\begin{lemma}
Let $X_i$ be the set of all tableaux arising in between steps $i-1$ and $i$ of $\lesh$. Let $X'_i$ be the set of all tableaux arising during $\sh$, when the $\ybox$ is between the $(i-1)$-st and $i$-th strips. Then $X_i = X'_i$.
\end{lemma}
\begin{proof}
Both sets consist of `punctured' semistandard tableaux of content $\beta$ and shape $\gamma^c / \alpha$, with ballot reading word, and where the $\ybox$ is between the $(i-1)$-st and $i$-th horizontal strips. (It is well-known that ballotness is preserved by jeu de taquin slides. Ballotness is also preserved during $\lesh$ by Theorem \ref{thm:ballotness}.) Both shuffling and evacuation-shuffling are invertible, so every such tableau arises in $X_i$ and $X'_i$.  \qed 
\end{proof}

We have $X_1 = \LRyb$ and we write $X_{t+1} = \LRby$, where $t$ is the length of $\beta$. \\

For $1 \leq i \leq t$, let $\ell_i : X_i \to X_{i+1}$ be the $i$-th step of $\lesh$. Let $s_i : X_{i+1} \to X_i$ be the jeu de taquin shuffle. We have the diagram

\[\xymatrix{
X_1 \ar@/_10pt/[r]_-{\ell_1} &
X_2 \ar@/_10pt/[r]_-{\ell_2} \ar@/_10pt/[l]_-{s_1}&
X_3 \ar@/_10pt/[r]_-{\ell_3} \ar@/_10pt/[l]_-{s_2} &
\cdots \ar@/_10pt/[r]_-{\ell_t} \ar@/_10pt/[l]_-{s_3} &
X_{t+1} \ar@/_10pt/[l]_-{s_t}.
}\]
By definition,
\[\omega = \sh \circ \lesh = s_1 \cdots s_t \circ \ell_t \cdots \ell_1.\] 

\begin{definition}
For $i=1,\ldots,t$, we define $$\omega_i=s_1s_2\cdots s_{i-1} (s_i \ell_i) s_{i-1}^{-1}\cdots s_2^{-1} s_1^{-1}.$$  
Note that we may factor $\omega$ as $$\omega=\omega_t \omega_{t-1}\cdots \omega_2\omega_1.$$
\end{definition}
\noindent Hence we have 
\begin{align}
\mathrm{sgn}(\omega) &\equiv \sum_{i=1}^t \mathrm{sgn}(\omega_i)\pmod{2}, \label{eqn:omega-i-inequality} \\
\mathrm{rlength}(\omega) &\leq \sum_{i=1}^t \mathrm{rlength}(\omega_i). \label{eqn:omega-i-sign}
\end{align}
%\begin{align}\sum_{i=1}^t \mathrm{rlength}(\omega_i) &\ge \mathrm{rlength}(\omega), \label{eqn:omega-i-inequality} \\ \sum_{i=1}^t \mathrm{sgn}(\omega_i) &\equiv\mathrm{sgn}(\omega) \pmod{2}. \label{eqn:omega-i-sign} \end{align}
It now suffices to determine the orbits of $\omega_i$, a computation interesting in its own right:
\begin{theorem}\label{thm:little-orbits}
Let $\mathrm{orb}_i$ be the set of orbits of $\omega_i$. Then:
\[\sum_{\mathcal{O} \in \mathrm{orb}_i} (|\mathcal{O}| - 1) = |K(\gamma^c/\alpha; \beta)(i)|.\]
In particular, $\mathrm{rlength}(\omega_i)=|K(\gamma^c/\alpha; \beta)(i)|$ and $\mathrm{sgn}(\omega_i) \equiv |K(\gamma^c/\alpha; \beta)(i)| \pmod 2$.
\end{theorem}
\begin{proof}
We use the bijection $\varphi_1$ of Theorem \ref{thm:generating-ktheory} to generate genomic tableaux. Let $T \in X_i$.

First, suppose $\ell_i$ applies a Phase 1 vertical slide, or a Phase 2 $\jump$ move consisting of all $\horiz$ steps. Both of these steps are equivalent to jeu de taquin slides, so in this case $\ell_i(T) = s_i^{-1}(T)$. Thus $T$ is a fixed point and does not contribute to the sum; it also does not generate a genomic tableau.

Next, it is easy to see that $\ell_i$ applies a Phase 1 move if and only if the following conditions hold:
\begin{itemize}
\item The suffix from $\ybox$ in $T$ is not tied for $(i-1,i)$, and
\item There is an $i$ before the $\ybox$ in the reading word of $T$.
\end{itemize}
The first condition implies that the $(i-1)$-st step of $\lesh$ was in Phase 1; the second rules out the transition to Phase 2.

We now analyze the orbits of $s_i \circ \ell_i$. If either of the above conditions fails, $\ell_i$ moves the $\ybox$ to the first $i$ after it in reading order for which the $(i,i+1)$ suffix is tied; then $s_i$ moves it to the start of that row of $i$'s. Otherwise, $s_i \circ \ell_i$ applies a Pieri move on the horizontal strip of $i$'s. Thus the $\ybox$ moves downwards in this strip, one row at a time, until one of the conditions fails.  

Since $\omega_i$ is a bijection, and the only two possible types of moves are moving down one row at a time on the $i$th strip or jumping upwards some number of rows, it follows that the orbit consists of a cycle, containing exactly one ``special jump'' up to the top of the cycle, and the rest downward Pieri moves.

Thus every orbit has a form similar to that of the Pieri case (Example \ref{exa:ktheory-pieri}): one step does not generate a genomic tableau; all other steps generate exactly one each. Thus during each orbit $\mathcal{O} \in \mathrm{orb}_i$, we generate $|\mathcal{O}| - 1$ genomic tableaux. Since every tableau of $K(\gamma^c/\alpha; \beta)(i)$ arises once in Phase 1, we are done.  \qed 
\end{proof}

Equations \eqref{eqn:recall-ineq} and \eqref{eqn:recall-parity} now follow from Theorem \ref{thm:little-orbits} and Equations \eqref{eqn:omega-i-inequality} and \eqref{eqn:omega-i-sign}.

\section{Orbits of \texorpdfstring{$\omega$}{w}}\label{sec:omega-orbits}

%%%%%%%%%%%%%%%%%%%%%%%%%%%
%%%%% ORBITS OF OMEGA %%%%%
%%%%%%%%%%%%%%%%%%%%%%%%%%%

\subsection{A stronger conjectured inequality}

For the first statement, numerical evidence suggests that, using either $\varphi_1$ or $\varphi_2$, the inequality in fact holds orbit-by-orbit (see Figure \ref{fig:numerical-evidence}):
\begin{conjecture} \label{conj:orbit-by-orbit}
Let $\mathcal{O} \subseteq \LRyb$ be an orbit of $\omega$. Let $K_1(\mathcal{O}), K_2(\mathcal{O})$ denote the sets of genomic tableaux occuring in this orbit in Phases 1 and 2 via the bijections $\varphi_1, \varphi_2$. Then
\[|K_i(\mathcal{O})| \geq |\mathcal{O}| - 1 \qquad \text{for } i = 1, 2.\]
Note that, by Corollary \ref{cor:antidiag-evacu-path}, it is sufficient to prove this for $\varphi_1$.
\end{conjecture}

\noindent We have verified Conjecture \ref{conj:orbit-by-orbit} for $n$ up to size $10$ and all $k$, $\alpha$, $\beta$, and $\gamma$. Below, we prove the conjecture in two special cases.

\begin{figure}[b]
\centering
\begin{tabular}{m{2cm} m{2cm} m{2cm} m{2cm}|c|c|c} 
\multicolumn{4}{c|}{Schubert problem} & $|\mathcal{O}|$ & $K_1(\mathcal{O})$ & $K_2(\mathcal{O})$ \\ \hline \vspace{0.1cm}
\multirow{3}{*}{$\alpha = {\tiny \yng(6,5,3,1)}$} &
\multirow{3}{*}{$\beta = {\tiny \yng(7,4,3,2)}$} &
\multirow{3}{*}{$\gamma = {\tiny \yng(5,5,4,2)}$} &
\multirow{3}{*}{$\rect = 6 \times 8$} &
38 & 52 & 51 \\
&&&& 23 & 31 & 28 \\
&&&& 10 & 9  & 13 \\ \hline \vspace{0.1cm}
\multirow{2}{*}{$\alpha = {\tiny \yng(2,2,1)}$} &
\multirow{2}{*}{$\beta = {\tiny \yng(3,1,1)}$} &
\multirow{2}{*}{$\gamma = {\tiny \yng(3,2)}$} &
\multirow{2}{*}{$\rect = 4 \times 4$} &
1 & 0 & 0 \\
&&&& 1 & 0 & 0
\end{tabular}
\caption{Examples of Schubert problems. For each problem, we list the size of each orbit $\mathcal{O}$ and the genomic tableaux $K_1(\mathcal{O}), K_2(\mathcal{O})$ corresponding to that orbit in Phases 1 and 2.}
\label{fig:numerical-evidence}
\end{figure}

\begin{remark}
The inequalities of equation \eqref{eqn:recall-ineq} and Conjecture \ref{conj:orbit-by-orbit} are tight bounds, since equality holds in the Pieri case and in several others.  Indeed, in the Pieri case $\omega$ has only one orbit and $|K| = |\mathcal{O}| - 1$.  Geometrically, this implies that the Schubert curve $S(\alpha,\beta,\gamma)$ is integral and has $\chi(\mathcal{O}_S) = 1$, so $S \cong \mathbb{P}^1$.
\end{remark}

\subsection{Fixed points of \texorpdfstring{$\omega$}{w}}

As a base case of Conjecture \ref{conj:orbit-by-orbit}, we characterize the fixed points of $\omega$.

\begin{proposition}\label{prop:fixed-points}
Let $T \in \LRyb$. The following are equivalent:
\begin{itemize}
\item[(i)] $\omega(T) = T$.
\item[(ii)] In the computation of $\lesh(T)$, neither bijection $\varphi_1, \varphi_2$ generates a genomic tableau.
\item[(iii)] The evacu-shuffle path of the $\ybox$ is connected.
\end{itemize}
\end{proposition}
\begin{proof}
It is easy to see that (ii) and (iii) are equivalent.  Moreover, if (iii) holds then the movements of the $\ybox$ are equivalent to jeu de taquin slides, so $\omega(T) = T$.  Thus (iii) implies (i).

To show (i) implies (ii), suppose first that the computation of $\lesh(T)$ involves a Pieri jump in Phase 1. Let $i$ be the index of the jump; the effect is that a single $i$ moved strictly up \emph{and} to the right. Since the horizontal strip of $i$'s is unaffected by the remaining steps of $\lesh$, the movement must be undone by the $\sh$. But the jeu de taquin slides can only either move a single $i$ down one row, or move a strip of $i$'s to the right. Neither is enough to undo the movement, so we conclude $\omega(T) \ne T$.

We have shown that if Phase 1 generates a genomic tableau, then $T$ is not a fixed point. By a similar argument, or by antidiagonal symmetry (Corollary \ref{cor:antidiag-evacu-path}), if Phase 2 generates a genomic tableau, then $\omega(T) \ne T$. This completes the proof.  \qed 
\end{proof}

\noindent One immediate corollary of this result is the following \emph{geometric} fact:
\begin{corollary}\label{cor:w=id}
Suppose $\omega$ acts on $\LRyb$ as the identity. Then $\Kabg = \eset$; it follows that the curve $S(\alpha, \beta, \gamma)$ is a disjoint union of $\mathbb{P}^1$'s, and the map $S \to \mathbb{P}^1$ is locally an isomorphism.
\end{corollary}

\begin{remark}
In general, a morphism of real algebraic curves $C \to D$, which is a covering map on real points, may have trivial real monodromy but be algebraically nontrivial (i.e. not a local isomorphism). Corollary \ref{cor:w=id} shows that this cannot occur for Schubert curves.
\end{remark}
\begin{proof}
If $\omega$ is the identity, Proposition \ref{prop:fixed-points} and Theorem \ref{thm:generating-ktheory} imply  $|\Kabg|=0$. Therefore $$\chi(\mathcal{O}_S) = |\LRyb| - |\Kabg| = |\LRyb|.$$ There are, moreover, exactly $|\LRyb|$ real connected components. It follows that $S$ has the desired form: using the notation of Proposition \ref{prop:numerics}, the inequalities
\[\eta(S)\geq\iota(S)\geq\dim_\mathbb{C} H^0(\mathcal{O}_S)\geq\chi(\mathcal{O}_S)\] become equalities. Note that $\dim_\mathbb{C} H^0(\mathcal{O}_S)$ is the number of $\mathbb{C}$-connected components of $S$. In particular each $\mathbb{C}$-connected component is irreducible, and of genus 0 because $\dim H^1(\mathcal{O}_S)=0$.  \qed 
\end{proof}

\noindent We also obtain a weaker form of the Orbits Conjecture:

\begin{corollary}\label{cor:order2-orbits}
For any orbit $\mathcal{O}$ of $\omega$,
\[|K_1(\mathcal{O})| + |K_2(\mathcal{O})| \geq |\mathcal{O}| - 1,\]
and if $|\mathcal{O}| \ne 1$ the inequality is strict. \end{corollary}
\begin{proof}
  This follows from Proposition \ref{prop:fixed-points}, since in each $\omega$-orbit that is not a fixed point, every step involves at least one genomic tableau generated in either Phase 1 or Phase 2. \qed  %, and each genomic tableau is counted twice in this fashion.  
\end{proof}
\noindent We think of this as an `order-2 approximation', since summing over the orbits gives
\[2 \cdot |\Kabg| \geq |\LRyb| - |\mathrm{orbits}(\omega)| = \mathrm{rlength}(\omega),\]
a weaker version of our Theorem \ref{thm:MainResult2}.

%\begin{corollary}
%  Let $\fix(\omega)=\{T\in \LRyb:\omega T=T\}$.  Then $$2 \cdot |K(\gamma^c/\alpha; \beta)|\ge |\LRyb|-|\fix(\omega)|.$$ 
%\end{corollary}
%
%\begin{proof}
%  This follows from Proposition \ref{prop:fixed-points}, since in each $\omega$-orbit that is not a fixed point, every step involves at least one genomic tableau generated in either Phase 1 or Phase 2, and each genomic tableau is counted twice in this fashion.
%\end{proof}

\subsection{When \texorpdfstring{$\beta$}{the rectification shape} has two rows}
%%%%%%%%%%%%%%%%%%%%%%%%%%%%%
%%%% \beta with two rows %%%%
%%%%%%%%%%%%%%%%%%%%%%%%%%%%%

In this section, we prove Conjecture \ref{conj:orbit-by-orbit} for $K_1(\mathcal{O})$ when $\beta$ has two rows. We note that the case where $\beta$ has one row (the Pieri Case) is trivial: equality holds for the (unique) orbit. See Example \ref{exa:ktheory-pieri}.

\begin{theorem}\label{thm:tworows}
Let $\beta$ have two rows. For an $\omega$-orbit $\mathcal{O} \subset \LRby$, let $K_1(\mathcal{O})$ be the set of ballot genomic tableaux occurring in $\mathcal{O}$ during Phase 1. Then
\begin{equation} \label{eqn:orbit-by-orbit-ineq}
|K_1(\mathcal{O})| \geq |\mathcal{O}| - 1.
\end{equation}
\end{theorem}

If the skew shape $\gamma^c/\alpha$ contains a column of height 3, then $\omega$ is the identity and $k=0$. For the remainder of this section, we assume every column of $\gamma^c/\alpha$ has height at most 2.

We use the following idea: consider the sub-shape of $\gamma^c/\alpha$ consisting of only its height-one columns. This shape consists of a disjoint union of row shapes. For a tableau $T \in \LRyb$ or $\LRby$, we will call the fillings of these row shapes the \newword{words} of $T$.

\begin{definition} \label{def:exceptional}
Let $(T,\ybox) \in \LRby$. We say that $T$ is \newword{exceptional} if the following holds:
\begin{itemize}
\item Every square of $T$ strictly above $\ybox$ contains a $1$.
\item From top to bottom, the words of $T$ are a sequence of all-$1$ words, followed by at most one `mixed' word containing $1$'s, $2$'s and/or $\ybox$, followed by a sequence of all-$2$ words.
\end{itemize}
\begin{example}
The following tableaux are exceptional:
\[T_1: \young(:::::::::::::1,::::::::11111:,:::::11122\x,::1122,222) \qquad T_2 : \young(:::::::::::1,:::::::::11,::::1112\x,::1222,222)\]
From top to bottom, the words of $T_1$ are $(1,11,11,12,22)$ and the words of $T_2$ are $(1,11,12\ybox,2,22)$.
\end{example}

\end{definition}
Note that $\LRby$, if nonempty, contains exactly one exceptional tableau (the second condition determines the words and the first determines placement of the $\ybox$). 
%Indeed, we may write $T$ down as follows:
%\begin{enumerate}
%\item[(1)] Fill all height-2 columns but the rightmost with ${\tiny \young(1,2)}$; place a $1$ in the top square of the rightmost height-2 column.
%\item[(2)] Fill the words of $T$ with $1$'s, starting from the top, until $T$ contains $\beta_1$ $1$'s.
%\item[(3)] Leave the rightmost square of the topmost incomplete word blank, and otherwise fill all the remaining squares with $2$'s.
%\item[(4)] Place $\ybox$ in the higher of the two empty squares and $2$ in the other. (If they are in the same row, place the $\ybox$ in the square from Step (1), at the end of the row).
%\end{enumerate}
\begin{proof}[Proof of Theorem]
As $\beta$ has two rows, $\lesh$ takes two steps. If both are Phase 1 Pieri moves, we have `gained' a Pieri move. If neither is, we have `lost' one. All other possibilities contribute 1 element to both $\mathcal{O}$ and $K_1(\mathcal{O})$, hence have no effect on the inequality. We will show that, in almost all cases, we `gain' a Pieri move between two successive `losses'. 

If $|\mathcal{O}| = 1$, we are done by Proposition \ref{prop:fixed-points}. Henceforth we assume $|\mathcal{O}| > 1$.

We divide the orbit into (disjoint) segments $T, \omega(T), \ldots, \omega^n(T)$, such that $\omega^{-1}(T) \to T$ and $\omega^{n-1}(T) \to \omega^n(T)$ have transition step $s < 3$, but the intermediate steps have $s=3$ (i.e. they remain in Phase 1). We will show that among all such segments, at most one contributes ${-1}$ to the inequality. The others contribute nonnegatively.

Within a segment, each intermediate $\lesh$ remains in Phase 1, hence involves at least one regular Pieri move (since the tableau is not fixed). If the last one does as well, or if some intermediate step involves two Pieri moves, then the entire segment contributes nonnegatively to the inequality. If not, we show:
\begin{lemma}
Suppose $\omega^{n-1}(T) \to \omega^n(T)$ does not involve a Pieri move, and every intermediate step involves exactly one. Then $\sh^{-1}(T)$ is exceptional.
\end{lemma}
Theorem \ref{thm:tworows} will follow since only one segment can begin with an exceptional tableau. \qed \end{proof}
\begin{proof}[Proof of Lemma]
By our hypotheses, every intermediate $\lesh$ step must consist of either $\vertical_1, \pieri_2$ or $\pieri_1, \vertical_2$.

First, we claim that if $\omega^i(T) \to \omega^{i+1}(T)$ consists of $\vertical_1, \pieri_2$, then every earlier step is also of this form, and every \emph{word} weakly above the $\ybox$ in $\omega^i(T)$ consists only of $1$'s. On the other hand, if $\omega^i(T) \to \omega^{i+1}(T)$ consists of $\pieri_1, \vertical_2$, we claim that every subsequent step is of this form, and every word strictly below $\ybox$ in $\omega^i(T)$ consists entirely of $2$'s.

For the first claim, we work backwards from $\omega^i(T)$ to $\omega^{i-1}(T)$. During the $\sh^{-1}$ step, the $\ybox$ slides one square down, then right; there must then be a $1$ directly above $\ybox$. If some row above $\ybox$ contains a $2$, $\lesh^{-1}$ must begin in Reverse Phase 1. (By construction, this will be the case as long as $i > 0$.) Hence $\lesh^{-1}$ consists of (Reverse) $\pieri_2$ and $\vertical_1$, as desired. For the claim about words, note that the (Reverse) $\pieri_2$ move will only move the $\ybox$ past words containing all $1$'s. Finally, if $i=0$, then $\lesh^{-1}$ begins in Reverse Phase 2 because there are no 2's in any word (in fact, any row) above $\ybox$ in $\sh^{-1}(T)$.

For the second claim, the argument is similar, only we work forward. The computation of $\lesh(\omega^i(T))$ terminates with $\ybox$ below a $2$; any words passed over by the $\ybox$ contain only $2$'s. During $\sh$, the $\ybox$ slides up and left, so it is above a $2$ in $\omega^{i+1}(T)$. If $i+1 < n-1$, then $\lesh$ will again have the form $\pieri_1, \vertical_2$. Finally, if $i+1 = n-1$, then $\lesh(\omega^{n-1}(T))$ must begin in Phase 2 (it can't begin with $\vertical_1$ since $\ybox$ is above a 2, and we have assumed it does not involve a regular Pieri move). Thus every row below the $\ybox$ contains only $2$'s.

We thus divide the segment into a first part, where $\lesh$ consists of $\vertical_1, \pieri_2$, and a second part, where $\lesh$ consists of $\pieri_1, \vertical_2$. Note that there can be a single `mixed' word in the tableau (if the second part begins with $\omega^i(T)$, this is the word to the right of the $\ybox$ in $\omega^i(T)$; in fact the $\ybox$ slides through this word during the $\sh$ step linking the two parts). We see, moreover, that all the non-mixed words remain unchanged from $\sh^{-1}(T)$ to $\omega^{n-1}(T)$.

Thus, from top to bottom, the words of $\sh^{-1}(T)$ are a (possibly empty) sequence of all-1 words, a single (possibly) `mixed' word containing $1$'s, $2$'s and/or the $\ybox$, followed by a (possibly empty) sequence of all-2 words. Thus $\sh^{-1}(T)$ is exceptional. \qed 
\end{proof}

In fact, our proof shows something slightly stronger: an orbit $\mathcal{O}$ for which $|K_1(\mathcal{O})| = |\mathcal{O}|-1$ is either a single fixed point, or is the unique orbit containing the exceptional tableau. All other orbits in fact satisfy $|K_1(\mathcal{O})| \geq |\mathcal{O}|$.

\section{Geometrical constructions}\label{sec:constructions}

%%%%%%%%%%%%%%%%%%%%%%%
%%%% CONSTRUCTIONS %%%%
%%%%%%%%%%%%%%%%%%%%%%%

We now give several families of values of $\alpha$, $\beta$, and $\gamma$ for which the Schubert curve $S(\alpha, \beta, \gamma)$ exhibits `extremal' numerical and geometrical properties.

\subsection{Schubert curves of high genus}

Recall that the arithmetic genus of a (connected) variety $S$ can be defined as \[g_a(S)=(-1)^{\dim S}(1 - \chi(\mathcal{O}_S)).\]
If $S$ is an integral curve, this is just $\dim_{\mathbb{C}} H^1(\mathcal{O}_S)$. (If $S$ is smooth, this is the usual genus of $S(\mathbb{C})$ as a topological space.)

In this section we construct a sequence of Schubert curves $S_t$, $t \geq 2$, for which $\omega$ has only one orbit, and so (by Proposition \ref{prop:numerics}) $S_t$ is integral. Moreover, we show that as $t \to \infty$, $g_a(S_t) \to \infty$ as well. In \cite{bib:Levinson}, the second author asked if Schubert curves are always smooth. K-theory does not in general detect singularities, but either possibility is interesting: that $S_t$ gives examples of singular Schubert curves for $t \gg 0$, or that it gives smooth Schubert curves of arbitrarily high genus.

As mentioned in the introduction, for our Schubert curves $S=S(\alpha,\beta,\gamma)$, we also have $$\chi(\mathcal{O}_S)=|\LRyb|-|K(\gamma^c/\alpha;\beta)|.$$  Therefore, if $S$ is connected (which is true if $\omega$ has one orbit), we have
\begin{eqnarray*}
  |\LRyb|-|K(\gamma^c/\alpha;\beta)| &=& \dim_{\mathbb{C}} H^0(\mathcal{O}_S)-\dim_{\mathbb{C}} H^1(\mathcal{O}_S)\\
  &=& 1-g_a(S).
\end{eqnarray*}
and so
\begin{equation}\label{eqn:connected} g_a(S)=|K(\gamma^c/\alpha;\beta)|-|\LRyb|+1. \end{equation}

We can now construct our family of high genus curves. Let $t \geq 3$ be a positive integer, and let
\begin{align*}
\alpha = \gamma &= (t,t-1,t-2,\ldots,2,1), \\
\beta &= (t+1,2,1^{t-2}).
\end{align*}
We work in the Grassmannian $G(t+1,\mathbb{C}^{2t+3})$, so $\rect$ has size $(t+1) \times (t+2)$, and $\gamma^c/\alpha$ is a staircase ribbon shape. (See Example \ref{exa:high-genus}.) We will call $\gamma^c/\alpha$ the \newword{staircase ribbon of size $t$}.

\begin{example}\label{exa:high-genus}
For $t=5$, two of the elements of $\LRyb$ are $$\young(:::::11,::::12,:::13,::14,:12,\x 5)\hspace{0.5cm}\text{  and  }\hspace{0.5cm} \young(:::::11,::::12,:::23,::\x 4,:11,15).$$  Each of these will be referred to as illustrations in our proof below.
\end{example}

\begin{proposition}\label{prop:high-genus}
With notation as above, $\omega$ has only one orbit. In particular, $S_t = S(\alpha,\beta,\gamma)$ is integral, and $g_a(S)=(t-1)(t-2).$
\end{proposition}

We break the proof of Proposition \ref{prop:high-genus}, into several intermediate lemmas. We first compute the cardinalities in question.

\begin{lemma}\label{lem:count-LR}
With notation as above,
$$|\LRyb|=2t(t-1).$$
\end{lemma}
\begin{proof}  
We sort the tableaux into two types: those for which the inner corners are all $1$ or $\ybox$, as in the first example in Example \ref{exa:high-genus}, and those for which there is an inner corner whose entry is greater than $1$, as in the second example.  We will refer to these as Type A and Type B tableaux.
  
  In a Type A tableau, the topmost outer corner must be a $1$ since the tableau is ballot.  Since there are a total of $t$ entries greater than $1$ and exactly $t+1$ outer corners, the remaining outer corners must be filled with $2,2,3,4,\ldots,t$, and all but the second $2$ must occur in that order.  There are $t-1$ possibilities for the position of the second $2$, and the remaining outer corners are determined.  The $\ybox$ can then be in any of the $t+1$ inner corners, and the remaining entries then must be filled with $1$'s.  This gives a total of $(t-1)(t+1)$ Type A tableaux.

   In a Type B tableau, ballotness forces exactly one inner corner to contain a $2$; among the outer corners, the topmost and one other contain $1$'s. The $\ybox$ must be above this second $1$; the remaining entries are determined. If the $2$ is in the lowest inner corner, there are $(t-1)$ choices for the $\ybox$. Each of the $(t-1)$ other placements of the $2$ gives $(t-2)$ choices for the $\ybox$, for a total of $(t-1)+(t-1)(t-2)=(t-1)^2$ Type B tableaux. \qed 
%  
%  All in all there are $2t(t-1)$ tableau, so $|\LRyb|=2t(t-1)$.
\end{proof}

\begin{lemma}\label{lem:count-K}
With notation as above, $$|K(\gamma^c/\alpha;\beta)|=3t^2-5t+1.$$
\end{lemma}

\begin{proof}
We count the ballot genomic tableaux having an extra $i$ for each $i$ separately. We use the description from Lemma \ref{lem:genomic-criterion}.

For $i=1$, the tableau must contain $(t+2)$ $1$'s.
%
%For $i=1$, we wish to count the semistandard tableaux on the staircase ribbon with content $\beta'=(t+2,2,1,1,\ldots,1)$ and with two non-adjacent marked $1$'s occurring consecutively in reading order such that removing either $1$ leaves a tableau with ballot reading word.
By semistandardness, we cannot have a $1$ in an outer corner besides the topmost outer corner.   Thus the entries \emph{larger} than $1$ fill all the outer corners except the topmost.  There are $t-1$ ways to place the second $2$, and all other entries are then determined by ballotness.  For each of these tableaux, there are then $t$ pairs of consecutive inner corners to mark as our chosen repeated $1$'s, and each of these satisfy the ballot condition on removal.  We therefore have $t(t-1)$ ballot genomic tableaux in this case.

For $i=2$, we wish to count for semistandard genomic tableaux with content $\beta''=(t+1,3,1,1,\ldots,1)$ and two marked $2$'s as above.  By semistandardness and ballotness, the topmost $2$ must be in the outer corner in the second row.  If the topmost $2$ is in the marked pair of $2$'s, then in order for the word to be ballot upon removing it, the next $2$ (necessarily the other marked $2$) must occur before the $3$.  The next $2$ therefore occurs in the third outer corner from the top, and by semistandardness and ballotness all entries larger than $2$ fill in the remaining outer corners, with the third $2$ in one of $t-1$ possible inner corners.  This gives $t-1$ genomic tableaux in this case. 

If the topmost $2$ is not in the marked pair, then the other two $2$'s must be in an inner and outer corner respectively which are not adjacent.  There are $(t-1)$ positions for the $2$ in the outer corner and then $(t-2)$ valid positions for the other $2$ for each of these choices, for a total of $(t-1)(t-2)$ possibilities in this case.  Thus we have a total of $(t-1)^2$ ballot genomic tableaux with two marked $2$'s.

Finally, if $i\ge 3$, it is easy to see by the semistandard and ballot conditions that the repeated $i$'s must be in the consecutive outer corners in the $i$th and $i+1$st rows from the top. For each $i$ there are then $t$ inner corners in which the second $2$ can be placed, and all other entries are determined.  It follows that there are a total of $t(t-2)$ ballot genomic tableaux in the case $i\ge 3$.

All in all, there are $t(t-1)+(t-1)^2+t(t-2)=3t^2-5t+1$ tableaux.  \qed 
\end{proof}

\begin{lemma}\label{lem:one-w-orbit}
With notation as above, $\omega:\LRyb\to \LRyb$ has only one orbit.
\end{lemma}

\begin{proof}
  By Lemma \ref{lem:count-LR}, it suffices to find an orbit of size $2t(t-1)$.

  We first introduce some new notation that will clarify the steps in our proof.  Let $A_{p,q}$ be the unique tableau having the $\ybox$ in the inner corner in the $p$th row from the top ($1 \leq p \leq t+1$) and with the $2$'s in the outer corners in the $2$nd and $q$th rows ($3 \leq q \leq t+1$).  Let $B_{p,q}$ be the tableau having the $\ybox$ in the $p$th row and the $2$'s in rows $2$ and $q$, but with the $2$ in the inner corner of row $q$. We have $2 \leq q \leq t+1$ and $1 \leq p \leq t$, and $q \ne p, p+1$.  (These are the Type A and Type B tableaux from Lemma \ref{lem:count-LR}.)
  
  We will show that, for any $q$ with $4\le q \le t+1$, we have \begin{equation}\label{eqn:q>3}\omega^{2t} A_{t+1,q}=A_{t+1,q-1},\end{equation} and for $q=3$ we have \begin{equation}\label{eqn:q=3}\omega^{2t}A_{t+1,3}=A_{t+1,t+1}.\end{equation}
%  Since $t\ge 3$, these tableaux $A_{t+1,q}$ exist, and so 
These facts together will show that the $\omega$-orbit of $A_{t+1,t+1}$ has length $2t(t-1)$.%, and so this is sufficient.

  To prove equations (\ref{eqn:q>3}) and (\ref{eqn:q=3}), let $q\in \mathbb{Z}$ such that $3\le q\le t+1$.  Starting with $A_{t+1,q}$, the first application of $\omega$ according to local evacuation shuffling and JDT consists of a single $\jump_1$ move to the very top row, followed by a JDT back to the inner corner.  Thus $\omega A_{t+1,q}=A_{1,q}$.

  Now, if $q$ is sufficiently large then $\lesh(A_{1,q})$ starts with $\pieri_1$ and $\pieri_2$, which results in the $\ybox$ being in row $q$ and the $2$ in the inner corner of row $2$.  There is then a single $\cpieri$ move and an upwards JDT slide.  Thus we have $$\omega^2A_{t+1,q}=\omega A_{1,q}=B_{q-2,2}.$$  The next move, to compute $\omega(B_{q-2,2})$, is $\vertical_1$ followed by a $\cpieri$ move to the $2$ in the inner corner and a $\horiz$ move that is undone by JDT to form $A_{2,q-1}$.  This pattern continues, with the next steps in the $\omega$-orbit being $$B_{q-3,3},A_{3,q-2},B_{q-4,4},A_{4,q-3},\ldots$$ until we reach $A_{r,q+1-r}$ where $r$ is such that $r$ and $q+1-r$ differ by either $2$ or $3$.  At this point, $\lesh (A_{r,q+1-r})$ starts with $\pieri_1$ and $\pieri_2$ as usual, but then the  $\cpieri$ leaves the $\ybox$ adjacent to the $2$, and the $\ybox$ and $2$ then switch via JDT.  Thus $$\omega A_{r,q+1-r}=A_{r+1,q+1-(r+1)}.$$
  
  After this special step with two consecutive Type A tableaux, the orbit resumes alternating between $A$'s and $B$'s with the first subscript of the $A$'s increasing by $1$ each time and the first subscript of the $B$'s decreasing, starting with $B_{q-r-2,r+2}$, and continuing until we reach $B_{1,q-1}$.  At this point we have applied $\omega$ exactly $2(q-2)$ times.
  
  Now, $\lesh(B_{1,q-1})$ consists of a single $\vertical_1$ followed by a long sequence of $\pieri$ moves, and the upwards JDT slide then results in the tableau $B_{t,q-1}$.  The orbit then alternates between $A$'s and $B$'s again in its usual manner until we reach $A_{v,q+t-v}$ where $v$ is such that $v$ and $q+t-v$ differ by either $2$ or $3$.  By the same reasoning as above, this maps to $A_{v+1,q+t-(v+1)}$ and the alternation pattern starts again, and continues until we reach $B_{q-2,t+1}$.  We have now applied $\omega$ an extra $2(t-q+2)-1$ times, for a total of $2t-1$ times.
  
  Finally, if $q\ge 4$ then $\omega B_{q-2,t+1}=A_{t+1,q-1}$ by the same reasoning as before, and so $\omega^{2t} A_{t+1,q}=A_{t+1,q-1}$.  If $q=3$, though, $\omega B_{q-2,t+1}=\omega B_{1,t+1}$, and so before the application of $\omega$ the $\ybox$ is in the top row and above a $1$, with the topmost $2$ in the row below that.  It follows that the local evacuation shuffle consists of a long sequence of $\pieri$ moves, and the JDT slide leaves us with $A_{t+1,t+1}$, as desired.  \qed 
\end{proof}

We now finish the proof of Proposition \ref{prop:high-genus}.

\begin{proof}[Proof of Proposition \ref{prop:high-genus}] 
  By Lemma \ref{lem:one-w-orbit} and Proposition \ref{prop:numerics}, $S_t = S(\alpha,\beta,\gamma)$ is integral.  It follows from Equation \ref{eqn:connected} and Lemmas \ref{lem:count-LR} and \ref{lem:count-K} that 
  \begin{align*}
  g(S)&=|K(\gamma^c/\alpha;\beta)|-|\LRyb|+1 \\ 
    &= 3t^2-5t+1-2t(t-1)+1\\
    &= (t-1)(t-2). 
  \end{align*}
  as desired.  \qed 
\end{proof}

\subsection{Curves with many connected components}

We next exhibit a sequence of Schubert curves $S(\alpha,\beta,\gamma)$ having arbitrarily many (complex) connected components. We use Corollary \ref{cor:w=id}, since in the case that $\omega$ is the identity map we know that the curve must consist of a disjoint union of $\mathbb{P}^1$'s. So, it suffices to find shapes $\alpha$, $\beta$, and $\gamma$ for which $\omega$ is the identity map and $\LRyb$ has many elements.

\begin{proposition}\label{prop:2x2hook}
  Suppose $\beta=(m,1,1,\ldots,1)$ is a hook shape and $\gamma^c/\alpha$ contains a $2\times 2$ square.  Then $\omega$ is the identity.
\end{proposition}

\begin{proof}
  Since the Littlewood-Richardson tableau are semistandard and ballot, the $\ybox$ must be in the upper left corner of the (necessarily unique) $2\times 2$ square in any tableau in $\LRyb$.  Moreover, there is a unique copy of each entry greater than $1$ and so these entries form a vertical strip.  Therefore, the entry just below the $\ybox$ must be a $1$, and so the $2\times 2$ square looks like $$\young(\x a,1b)$$ for some $a$ and $b$.  We also have $a<b$ since the tableau is semistandard, and so in particular $b>1$.
 
  Now, we wish to show that any such filling maps to itself under $\omega$.  The first step in $\lesh$ must be $\vertical_1$.  At this step, since $b>1$ and the reading word is ballot, the unique $2$ in the tableau must occur after $\ybox$ in the reading word, and so the transition step is $s=2$.

  At this step, since the entries greater than $1$ appear in reverse reading order by ballotness and each occur exactly once, the smallest $k$ for which the $(k,k+1)$ suffix not tied is $k=b$.  It follows that the $\ybox$ switches with the $b$ as its only Phase 2 move; after this point the remaining $(i,i+1)$-suffixes for $i\ge b$ are empty and therefore tied.

$$\raisebox{-5pt}{\young(\x a,1b)}\xrightarrow{\vertical_1} \raisebox{-5pt}{\young(1a,\x b)} \xrightarrow{\horiz_2} \raisebox{-5pt}{\young(1a,b\x)}\xrightarrow{\text{JDT}} \raisebox{-5pt}{\young(\x a,1b)}$$
  
  Finally, we perform a JDT slide to move the $\ybox$ past the tableau, and we see that all entries are restored to their original position, as shown above.  \qed  
\end{proof}

We will now construct our curve in the Grassmannian $\mathrm{Gr}(m+1,\mathbb{C}^{2m+2})$ so that our shapes fill an $(m+1)\times (m+1)$ rectangle.  

\begin{proposition}
 Let $m$ be a positive integer.  Let $\beta=(m,1,1)$, let $\alpha=(m,m-1,m-2,\ldots,2)$, and let $\gamma^c=(m+1,m,m-1\ldots,4,3,2,2)$.  Then $S(\alpha,\beta,\gamma)$ consists of a disjoint union of exactly $m-1$ copies of $\mathbb{P}^1$.
\end{proposition}

\begin{proof}
The shape $\gamma^c/\alpha$ consists of a single $2\times 2$ square in the lower left corner plus $m-1$ disconnected boxes to the northeast.  Thus we have $\omega=\id$ by Proposition \ref{prop:2x2hook}, and by Corollary \ref{cor:w=id}, it follows that $S(\alpha,\beta,\gamma)$ is a disjoint union of exactly $|\LRyb|$ copies of $\mathbb{P}^1$.

We claim that $|\LRyb|=m-1$.  Indeed, since $\beta=(m,1,1)$, we wish to count ballot fillings that have one $2$, one $3$, and the rest $1$'s.  Since the $2\times 2$ box is at the lower left corner, the $3$ must be in the lower right corner of the $2\times 2$ box by the ballot and semistandard conditions.  It is easy to check that the $2$ can be in any of the remaining squares except the top row or in the leftmost column of the skew shape.  The positions of the $2$ and $3$ determine the tableau, so we have a total of $m-1$ Littlewood-Richardson tableaux. \qed 
\end{proof}

\section{Conjectures}\label{sec:conjectures}

%%%%%%%%%%%%%%%%%%%%%
%%%% CONJECTURES %%%%
%%%%%%%%%%%%%%%%%%%%%

We recall the conjectural `orbit-by-orbit' inequality:
\begin{conjecture}[Conjecture \ref{conj:orbit-by-orbit}]
Let $\mathcal{O} \subseteq \LRyb$ be an orbit of $\omega$. Let $K_1(\mathcal{O}), K_2(\mathcal{O})$ denote the sets of genomic tableaux occuring in this orbit in Phases 1 and 2 (via the bijections $\varphi_1, \varphi_2$). Then
\[|K_i(\mathcal{O})| \geq |\mathcal{O}| - 1 \qquad (\text{for } i = 1, 2).\]
Note that, by Corollary \ref{cor:antidiag-evacu-path}, it is sufficient to prove this for $\varphi_1$.
\end{conjecture}
We have proven Conjecture \ref{conj:orbit-by-orbit} in certain cases, but do not know a proof in general. This conjecture suggests that there is additional combinatorial structure in the complex curve $S(\mathbb{C})$ -- in particular its irreducible decomposition and, for each irreducible component $S' \subset S(\mathbb{C})$, the number of real connected components of $S'(\mathbb{R})$. We have in mind the following observation:

\begin{proposition} \label{prop:ram-pts}
Suppose $S$ is smooth and let $R = R(\alpha,\beta,\gamma) \subset S(\alpha,\beta,\gamma)$ be the ramification locus of the map $f: S \to \mathbb{P}^1$ of Theorem \ref{thm:intro-2}. Then $R$ is a union of complex conjugate pairs of points and, counted with multiplicity,
\[\tfrac{1}{2}|R(\alpha,\beta,\gamma)| = |\Kabg|.\]
\end{proposition}
\begin{proof}
The quantity $\tfrac{1}{2}R$ is the number (with multiplicity) of complex conjugate pairs of ramification points because $f$ is defined over $\mathbb{R}$ but none of its ramification points are real. The equation then follows from the Riemann-Hurwitz formula, which states
\[\chi(\mathcal{O}_S) = (\deg f) \cdot \chi(\mathcal{O}_{\mathbb{P}^1}) - \tfrac{1}{2} \deg R.\]
Note that $\deg f = |\LRyb|$, that $\chi(\mathcal{O}_{\mathbb{P}^1}) = 1$, and that $\chi(\mathcal{O}_S) = |\LRyb| - |\Kabg|$.  \qed 
\end{proof}
Proposition \ref{prop:ram-pts} suggests that genomic tableaux be used to index \emph{complex conjugate pairs} of ramification points.

\begin{question}
Is it possible to assign, to each complex conjugate pair of ramification points in $R(\alpha,\beta,\gamma)$, a genomic tableau from $K(\gamma^c/\alpha;\beta)$?
\end{question}

Conjecture \ref{conj:orbit-by-orbit} then suggests assigning to each ramification point $p \in R$ an arc on some component of $S(\mathbb{R})$ -- ideally on the same irreducible component as $p$ -- compatibly with the labeling by genomic tableaux and the bijections $\varphi_i$. Such an assignment would further relate the real and complex topology of $S$. For instance:

\begin{question}
Suppose $T \in \LRyb$ is an $\omega$-fixed point. Let $S' \subseteq S$ be the irreducible component containing $T$. Must $S'$ be a copy of $\mathbb{P}^1$, mapping (via $f$) to $\mathbb{P}^1$ with degree 1?
\end{question}
The converse is true: if some component $S'$ maps isomorphically to $\mathbb{P}^1$, then $S'(\mathbb{R}) \cap f^{-1}(0)$ corresponds to an $\omega$-fixed point under the identification of $f^{-1}(0)$ with $\LRyb$. On the other hand, we have shown (Proposition \ref{prop:fixed-points}) that if \emph{every} $T$ is a fixed point, that is, $\omega$ is the identity, then $S$ is indeed a disjoint union of $\mathbb{P}^1$'s, each mapping isomorphically under $f$. 

\begin{question}
Let $\mathcal{O} \subseteq \LRyb$ be an orbit such that $K_i(\mathcal{O}) = |\mathcal{O}| - 1$ for $i=1,2$. Let $S' \subseteq S$ be the irreducible component containing $\mathcal{O}$. Must $S'$ be a copy of $\mathbb{P}^1$, mapping to $\mathbb{P}^1$ with degree $|\mathcal{O}|$?
\end{question}
% I think the answer to this question is almost certainly "no".

If the global inequality \eqref{eqn:ktheory-ineq} is replaced by an equality (and is then true of every orbit), it is possible to show that this is true, i.e. that $S$ is a disjoint union of $\mathbb{P}^1$'s, each mapping to $\mathbb{P}^1$ with the appropriate degree -- in particular, in the Pieri Case. On the other hand, if a single irreducible component $S'$ contains a number of ramification points equal to $(\deg f|_{S'}) - 1$, then the Riemann-Hurwitz formula implies that $g(S') = 0$, i.e. $S' \cong \mathbb{P}^1$ and $S'(\mathbb{R})$ has only one connected component.

Finally, although we have only defined \emph{local} evacuation-shuffling for Littlewood-Richardson tableaux, the evacuation-shuffle $\esh$ is defined on \emph{all} tableaux $(\ybox,T)$ as the conjugation of shuffling by rectification. %Since we have been interested on the action of $\esh$ on dual equivalence classes, we have studied only their highest-weight representatives, i.e. Littlewood-Richardson tableaux.
Our results do yield local algorithms for certain other classes of tableaux, such as \emph{lowest}-weight semistandard tableaux, via straightforward alterations to $\lesh$. (For lowest-weight semistandard tableaux, the local algorithm resembles a rotated version of $\lesh^{-1}$.) It would be interesting to understand the actions of $\esh$ and $\omega$ on arbitrary representatives of dual equivalence classes, and on semistandard tableaux in general. We may be more precise:

\begin{conjecture}
Let $T$ be \textbf{any} (semi)standard skew tableau and $\ybox$ an inner co-corner of $T$. There exists a local algorithm for computing $\esh(\ybox,T)$, which does not require rectifying the tableau, such that:
\begin{itemize}
\item[(i)] Each step consists of exchanging the $\ybox$ with an entry of $T$, of weakly increasing value.
\item[(ii)] The slide equivalence class of $T$ is preserved throughout the algorithm.
\item[(iii)] The algorithm specializes to jeu de taquin (if $T$ is of straight shape) and $\lesh$ (if $T$ is ballot).
\end{itemize}
Each step should correspond (by conjugating with rectification) to a jeu de taquin slide of $\ybox$ through the rectification $\rectify(\ybox,T)$. \end{conjecture}
It would also be interesting to investigate how such algorithms might relate to K-theoretic Schubert calculus.

For a \emph{straight-shape} tableau $T$ that is \emph{not} highest-weight, the shuffle path of the $\ybox$ is just the path given by jeu de taquin slides through $T$. It would be interesting to find a generalization of the $s$-decomposition that describes this shuffle path, and that gives rise to a local algorithm on any skew tableau $T'$ whose rectification is $T$. 

We may also ask analogous questions for computing $\esh(S,T)$ locally, where both $S$ and $T$ may have more than one box.

\end{document}